\documentclass[onefignum,onetabnum]{siamart190516}

\usepackage{lipsum}
\usepackage{amsfonts}
\usepackage{graphicx}
\usepackage{epstopdf}
\usepackage{algorithmic}
\usepackage{amsopn}
\DeclareMathOperator{\diag}{diag}

\ifpdf
  \DeclareGraphicsExtensions{.eps,.pdf,.png,.jpg}
\else
  \DeclareGraphicsExtensions{.eps}
\fi

\newsiamremark{remark}{Remark}
\newsiamremark{hypothesis}{Hypothesis}
\crefname{hypothesis}{Hypothesis}{Hypotheses}
\newsiamthm{claim}{Claim}

\headers{Regularized inexact interior-point method}{J. Cornelis \& W. Vanroose }

\title{Convergence analysis of a regularized inexact interior-point method for linear programming problems\thanks{Submitted to the editors \today.
\funding{This work was funded in part by VLAIO Baekeland project HBC.2020.2867 and Motulus BV.}}}
\author{Jeffrey Cornelis\thanks{Department of Mathematics, University of Antwerp, Middelheimlaan 1, 2020 Antwerp, Belgium. (\email{jeffrey.cornelis@uantwerp.be})}
        \and Wim Vanroose\thanks{Department of Mathematics, University of Antwerp, Middelheimlaan 1, 2020 Antwerp, Belgium. (\email{wim.vanroose@uantwerp.be})}
	}

\begin{document}

\maketitle

\begin{abstract}
Interior-point methods for linear programming problems require the repeated solution of a linear system of equations. Solving these linear systems is non-trivial due to the severe ill-conditioning of the matrices towards convergence.  This issue can be alleviated by incorporating suitable regularization terms in the linear programming problem. Regularization also allows us to efficiently handle rank deficient constraint matrices.  We provide a convergence analysis of a regularized inexact interior-point method. The term `inexact' refers to the fact that we do not need to compute the true solution of the linear system of equations, only an approximation thereof. The formulation of the algorithm is sufficiently general such that specialized linear algebra routines developed in other work on inexact interior-point methods can also be incorporated in our regularized framework. In this work, we exploit the inexactness by using a mixed-precision solver for the linear system of equations. More specifically, we perform a Cholesky factorization in IEEE single precision and use it as a preconditioner for the Conjugate Gradient method. Numerical experiments illustrate the benefits of this approach applied to linear programming problems with a dense constraint matrix.   
\end{abstract}



\section{Introduction}

We consider the standard form linear programming problem
\begin{equation} \label{eq:primal}
\min_{x\in\mathbb{R}^n} c^T x \qquad \text{subject to} \qquad Ax = b,\, x\geq 0 
\end{equation}
with $x,c \in \mathbb{R}^n, b\in \mathbb{R}^m$ and $A\in\mathbb{R}^{m \times n}$ and its dual formulation 
\begin{equation}\label{eq:dual}
\max_{y\in\mathbb{R}^m} b^T y \qquad \text{subject to} \qquad A^T y + z = c,\, z\geq 0
\end{equation}
with $y\in\mathbb{R}^m$ and $z\in\mathbb{R}^n$. We refer to $x$ as the primal variables and $(y,z)$ as the dual variables. 
We do not need to assume that $A$ has full rank, which is one of the benefits of the approach proposed in the current paper. Inequalities should be interpreted component-wise.

The Karush-Kuhn-Tucker (KKT) conditions of the primal-dual pair of linear programming problems \cref{eq:primal} and \cref{eq:dual} are given by the non-linear system of equations
\begin{align}
Ax - b &= 0, \label{eq:primal_kkt}\\ 
A^T y + z - c &= 0,\label{eq:dual_kkt}\\
x_i z_i &= 0, \qquad (i = 1,\ldots,n) \label{eq:complementarity}\\
(x,z) &\geq 0. \label{eq:nonneg}
\end{align}
We assume that there exists a strictly feasible point, i.e. a point $(x,z)>0$ that satisfies conditions \cref{eq:primal_kkt,eq:dual_kkt}. In this case it is well known that the KKT conditions are necessary and sufficient for $x$ the be a solution of the primal linear programming problem \cref{eq:primal} and for $(y,z)$ to solve the dual problem \cref{eq:dual}, see for instance \cite{boyd2004convex,nocedal2006numerical}.

Primal-dual interior-point methods generally compute a perturbed Newton direction to the KKT system above, where the complementarity condition \cref{eq:complementarity} is relaxed by $x_i z_i = \tau$ for some positive scalar $\tau >0$. The solution $(x(\tau),y(\tau),z(\tau))$ obtained by solving equations \cref{eq:primal_kkt}, \cref{eq:dual_kkt}, \cref{eq:nonneg} and $x_i z_i=\tau$ for all $i$, guides us towards the true solution of the KKT equations for $\tau\rightarrow 0$. This set of points parametrized by $\tau$ is referred to in literature as the central path. Most primal-dual interior-point methods compute a sequence of iterates in some neighborhood of this central path and require that all iterates remain strictly feasible with respect to the constraint $(x,z)>0$. We refer the reader to \cite{wright1997primal} for an excellent general reference on primal-dual interior-point methods.
In this work we shall consider $\tau = \beta_1 x^T z / n$ where $\beta_1 \in(0,1)$ is called the centering parameter. In a typical primal-dual infeasible interior-point method we have that in each iteration a Newton-like direction is computed as:
\begin{equation} \label{eq:newton_eq}
\begin{pmatrix} A & 0 & 0 \\ 0 & A^T & I \\ Z & 0 & X \end{pmatrix} \begin{pmatrix} \Delta x \\ \Delta y \\ \Delta z \end{pmatrix} = 
\begin{pmatrix} b - A x \\ c - A^T y - z \\ \tau e - X z \end{pmatrix}.
\end{equation}
Here, we use $X = \diag(x)$ and $Z = \diag(z)$ to denote diagonal matrices containing the elements of $x$ and $z$ respectively. In addition we denote $e = (1,1,\ldots,1)^T \in \mathbb{R}^n$ the vector of length $n$ containing all ones, and $0$ and $I$ the zero-matrix and identity matrix respectively, where the dimensions should be clear from the context. Note that the matrix in \cref{eq:newton_eq} is simply the Jacobian matrix of the non-linear system of equations defined by \cref{eq:primal_kkt}--\cref{eq:complementarity}. These equations are only `mildly' non-linear in the sense that the only non-linearity is due to the complementarity condition.

So-called inexact interior-point methods \cite{korzak2000convergence,gondzio2013convergence,baryamureeba2000convergence,bellavia1998inexact,monteiro2003convergence,al2009convergence} do not require that \cref{eq:newton_eq} is solved exactly (i.e. to full machine precision accuracy). This type of interior-point method is for instance useful when we want to use some Krylov subspace method to solve a suitable reformulation of the linear system of equations. Two main approaches exist. The first consists of transforming \cref{eq:newton_eq} to a linear system of equations with a symmetric indefinite matrix, which is referred to as the augmented system approach. In this case we can use a symmetric indefinite Krylov subspace method, such as SYMMLQ or MINRES \cite{paige1975solution}. Alternatively, we can also reformulate \cref{eq:newton_eq} to a linear system with a symmetric positive definite matrix, in which case Conjugate Gradients (CG) \cite{hestenes1952methods} is preferred. This reformulation is referred to as the normal equations approach. We provide some more details on these two different approaches and how we can efficiently solve them in \cref{sec:reformulations}. For more general information on Krylov subspace methods we refer to \cite{liesen2013krylov,saad2003iterative}.

When the matrix $A$ does not have full row rank then it holds that the matrix \cref{eq:newton_eq} is singular, which is also true for the reduced linear systems. Moreover, even when $A$ has full rank, solving a suitable reformulation of the linear system \cref{eq:newton_eq} might lead to numerical difficulties. The main reason for this is the ill-conditioning of the matrices towards convergence caused by complementarity of the variables $x$ and $z$, which is a well-known result in the literature. This is also the main difficulty discussed in a lot of the inexact interior-point methods based on Krylov subspace methods. These Krylov subspace methods need a very good preconditioner to make reasonable progress towards the solution, since the rate of convergence is related to the condition number of the matrix. 

To alleviate these issues, several papers have appeared in the literature that start from a regularized formulation of the linear programming problem, or introduce regularization terms directly in the linear systems \cite{friedlander2012primal,altman1999regularized,saunders1996solving,gondzio2012matrix}. These references either assume that the linear systems are solved exactly or provide no formal convergence analysis at all. As far as we known, no attempt has been made to provide a convergence analysis of a primal-dual infeasible interior-point method that solves a \textit{regularized} linear system \textit{inexactly}. The current manuscript fills this gap.

We keep the description of the newly proposed algorithm sufficiently general such that specialized linear algebra routines developed in previous work on inexact interior-point methods can also be incorporated in our regularized framework. Krylov subspace methods with specialized preconditioners \cite{gondzio2013convergence,cui2019implementation,monteiro2003convergence,al2009convergence} are especially well-suited. 
Note that Krylov subspace methods are in general designed for large and sparse matrices. However, we will use them in a slightly different context for our numerical experiments. More specifically we will consider test-problems with a dense constraint matrix $A$ which requires the approximate solution of a dense linear system of equations in each iteration of the interior-point method. We will solve these systems (approximately) by performing a Cholesky factorization in IEEE single precision and then perform -- if necessary -- a few iterations of the Conjugate Gradient method, preconditioned with the (dense) triangular Cholesky factors. In this way we can interpret the CG iterates as a form of iterative refinement, rather than a true iterative solution method. Alternatively this can be seen as some kind of hybrid between a direct method and an iterative method. Similar techniques have recently been proposed to accelerate the solution of linear systems of equations using lower precision arithmetic \cite{carson2017new,carson2018accelerating,higham2021exploiting,higham2019squeezing}. More details will be provided in \cref{sec:implementation}.

The rest of the paper is organized as follows. In \cref{sec:description} we give a detailed description of the new algorithm and provide a convergence analysis. Next, we discuss some implementation details in \cref{sec:implementation}. In \cref{sec:num_ex} we report a number of experiments which illustrate the benefit of the proposed algorithm. Lastly, this work is concluded and an outlook is given in \cref{sec:conclusions}.

\section{Description of the algorithm} \label{sec:description}
To derive our algorithm we use the regularized linear programming problem defined in \cite{friedlander2012primal}:
\begin{equation}\label{eq:regularized_lp}
\min_{(x,w)}  c^T x +  \frac{\rho}{2}\| x - x^k \|^2 + \frac{\delta}{2}\| w + y^k \|^2  \qquad \text{subject to} \qquad Ax +\delta w = b,\, x\geq 0. 
\end{equation} 
Here, $\rho \geq 0$ and $\delta \geq 0$ are regularization parameters, $x^k$ and $y^k$ are the current estimates of the primal and dual solutions and $w\in\mathbb{R}^m$ is an auxiliary variable. Here and throughout an upper-index refers to a certain iteration number while a lower-index refers to a particular component of a vector. For example $x_i^k$ denotes the $i$-th component of the $k$-th iterate $x^k$. The norm $\|\cdot\|$ denotes the Euclidean norm.

Note that we recover \cref{eq:primal} for the choice $\rho = \delta = 0$.
Our algorithm will coincide with the one presented in \cite{korzak2000convergence} for these choices of regularization parameters. With this in mind, we purposefully keep the description of our algorithm as close as possible to \cite{korzak2000convergence}, such that the subtle difference between the two is apparent. 

The KKT conditions for \cref{eq:regularized_lp} are given by
\begin{align*}
Ax + \delta w - b &= 0, \\ 
A^T y + z - c - \rho (x - x^k) &= 0,\\
x_i z_i &= 0, \qquad (i = 1,\ldots,n)\\
\delta (w + y^k) - \delta y &= 0, \\
(x,z) &\geq 0.
\end{align*}
From the second to last equation we immediately get the relation $w = y - y^k$ when $\delta > 0$. Elimination of this variable from the KKT equations and relaxing the complementarity condition $x_i z_i = \tau$, we get the following Newton-like update
\begin{equation} \label{eq:newton_eq_reg}
\begin{pmatrix} A & \delta I & 0 \\ -\rho I & A^T & I \\ Z & 0 & X \end{pmatrix} \begin{pmatrix} \Delta x \\ \Delta y \\ \Delta z \end{pmatrix} =  \begin{pmatrix} \xi \\ \zeta \\ \eta \end{pmatrix},
\end{equation}
with $\xi =  b - A x - \delta(y - y^k),\zeta = c +\rho(x - x^k)- A^T y - z$ and $\eta=\tau e - X z$. Now if we choose $x = x^k, y = y^k$ and $z = z^k$, then the right-hand sides in \cref{eq:newton_eq,eq:newton_eq_reg} coincide and the Jacobian matrices only differ by the terms $\delta I$ and $\rho I$. 

We modify the algorithm in \cite{korzak2000convergence} such that in stead of solving linear systems of the form \cref{eq:newton_eq} inexactly, we now solve linear systems of the form \cref{eq:newton_eq_reg} inexactly. Note that the analysis in \cite{friedlander2012primal} assumes that the linear systems are solved exactly and that the analysis in \cite{korzak2000convergence} does not include any form of regularization. Hence, the algorithm in the current manuscript can be seen as a combination of \cite{friedlander2012primal} and \cite{korzak2000convergence}. 

Suppose we have some iterate $(x^k,y^k,z^k)\in\mathbb{R}^{n + m + n}$. We refer to this point as an $(\epsilon,\epsilon_p,\epsilon_d)$-solution if the following conditions hold: 
\begin{equation}\label{eq:def_conv}
(x^k,z^k)\geq 0,\, (x^k)^Tz^k \leq \epsilon,\, \|Ax^k - b\| \leq \epsilon_p,\,\|A^T y^k + z^k - c\| \leq \epsilon_d. 
\end{equation}
The goal is to find an $(\epsilon,\epsilon_p,\epsilon_d)$-solution to the pair of linear programming problems \cref{eq:primal,eq:dual} or the determine they are likely to be infeasible when $||(x^k,z^k)||_1>\omega$, for some sufficiently large number $\omega$. Some comments on the latter stopping-criterion are given later in this section, see \cref{thm:omega_stop}.

We make sure all iterates of the interior-point method remain in the following central-path neighborhood, originally proposed by Kojima, Megiddo, Mizuno in \cite{kojima1993primal}:
\begin{align} \label{eq:N}
\mathcal{N} = \{(x,y,z) \in\mathbb{R}^{n + m + n}\,:\, &(x,z)>0,  \\ 
																											 &x_i z_i \geq \gamma x^T z/n \quad (i = 1,\ldots,n \nonumber), \\ 
																											 & x^T z \geq \gamma_p \|Ax - b \| \text{ or }  \|Ax - b \| \leq \epsilon_p, \nonumber \\
																									     & x^T z \geq \gamma_d \|A^T y + z - c\| \text{ or }  \|A^T y + z - c\| \leq \epsilon_d, \nonumber \}
\end{align}
for some constants $\gamma_p,\gamma_d>0$ and $\gamma \in (0,1)$. The same neighborhood is also used in \cite{korzak2000convergence}. When $\gamma,\gamma_p$ and $\gamma_d$ are all very small, say all equal to $10^{-8}$, then these conditions are very loose.

We now have all ingredients needed to formulate the regularized inexact interior-point method, see \cref{alg:KMM}. Suppose we have some iterate $(x^k,y^k,z^k)$ that is not yet an $(\epsilon,\epsilon_p,\epsilon_d)$-solution or does not satisfy $||(x^k,z^k)||_1>\omega$. 
Then we compute an approximation $(\Delta x^k,\Delta y^k,\Delta z^k)$ to the Newton direction \cref{eq:newton_eq_reg} for the choice $(x,y,z) = (x^k,y^k,z^k)$, see \cref{eq:update_eq}. The accuracy of the approximation is described using the norm of the \textit{residual components}, as shown in \cref{eq:acc_r,eq:acc_s}. Next, suitable step-lengths $\alpha^k_p$ and $\alpha^k_d$ are computed such that the new iterates $x^{k+1} = x^k + \alpha^k_p\Delta x^k$ and $(y^{k+1},z^{k+1}) = (y^{k},z^k) + \alpha^k_d(\Delta y^k, \Delta z^k)$ remain in the central path neighborhood $\mathcal{N}$ and satisfy an additional descent condition \cref{eq:descent}. Since $\beta_3>\beta_2$ we can always choose $\alpha_p^k = \alpha_d^k = \bar{\alpha}^k$. However, taking different step-lengths for the primal and dual variables allows us to take larger steps whenever possible and thus make more progress towards the solution. 

We now turn our attention to a convergence analysis of \cref{alg:KMM}. Our analysis closely resembles the analysis in \cite{korzak2000convergence}, which in its turn closely resembles the analysis of the exact counterpart of Kojima, Megiddo, Mizuno in \cite{kojima1993primal}. 

Let us denote the following functions for $i = 1,2,\ldots,n$:
\begin{align}
f_i^k (\alpha) &= (x_i^k + \alpha \Delta x^k_i)(z_i^k + \alpha \Delta z^k_i) - \gamma (x^k + \alpha\Delta x^k)^T(z^k + \alpha\Delta z^k)/n, \label{eq:def_f} \\
g_p^k(\alpha)  &= (x^k + \alpha\Delta x^k)^T(z^k + \alpha\Delta z^k) - \gamma_p \|A(x^k + \alpha\Delta x^k) - b \|, \label{eq:def_g_primal} \\
g_d^k(\alpha)  &= (x^k + \alpha\Delta x^k)^T(z^k + \alpha\Delta z^k) - \gamma_d \|A^T(y^k + \alpha\Delta y^k)+  z^k + \alpha\Delta z^k - c \|,\label{eq:def_g_dual} \\ 
h^k(\alpha)    &= (1-\alpha(1 - \beta_2)) (x^{k})^T z^k -(x^k + \alpha\Delta x^k)^T(z^k + \alpha\Delta z^k). \label{eq:def_h}
\end{align}
Furthermore, let us denote $\dot{\alpha}^k$ the largest value in $[0,1]$ such that for all $i=1,\ldots,n$:
\begin{align}
&f_i^k(\alpha) \geq 0,\, h^k(\alpha) \geq 0, \label{eq:ineq1} \\
&g^k_p(\alpha) \geq 0\text{ or } \|A(x^k + \alpha\Delta x^k) - b \| \leq \epsilon_p   \label{eq:ineq2}\\
&g^k_d(\alpha) \geq 0\text{ or } \|A^T(y^k + \alpha\Delta y^k)+  z^k + \alpha\Delta z^k - c \| \leq \epsilon_d,   \label{eq:ineq3}
\end{align}
for all $\alpha \in[0,\dot{\alpha}^k]$.

\begin{algorithm}
\caption{Regularized inexact interior-point method}
\label{alg:KMM}
\begin{algorithmic}[1]
\STATE{Choose $\omega\gg 1$ large, $\epsilon,\epsilon_p,\epsilon_d, \gamma_p,\gamma_d>0\,,\gamma \in (0,1)$ and $0<\beta_1<\beta_2<\beta_3<1$.}
\STATE{Choose regularization parameter $\delta$ and $\rho$ sufficiently small.}
\STATE{Choose parameters $\tau_1,\tau_2 \in (0,1]$ such that $\beta_1 + \tau_1 - 1 > 0$ and $\beta_1 + \tau_2 - 1 > 0$.}
\STATE{Compute initial point $(x^0,y^0,z^0)\in\mathcal{N}$, with $\mathcal{N}$ defined by \cref{eq:N}.}
\FOR{$k = 0,1,2,\ldots$}
\STATE{If $(x^k,y^k,z^k)$ is an $(\epsilon,\epsilon_p,\epsilon_d)$-solution or $||(x^k,z^k)||_1 > \omega$: \textbf{STOP}.}
\STATE{Compute an \textit{inexact} search direction $(\Delta x^k,\Delta y^k,\Delta z^k)$ that satisfies:
\begin{equation} \label{eq:update_eq}
\begin{pmatrix} A & \delta I & 0 \\ -\rho I & A^T & I \\ Z^k & 0 & X^k \end{pmatrix} \begin{pmatrix} \Delta x^k \\ \Delta y^k \\ \Delta z^k \end{pmatrix} = 
\begin{pmatrix} b - A x^k \\ c - A^T y^k - z^k \\ \beta_1\mu_k e - X^k z^k \end{pmatrix} + \begin{pmatrix}r^k \\ s^k\\ 0\end{pmatrix}
\end{equation}
with $\mu_k = \frac{\left(x^k\right)^T \left(z^k\right)}{n}$ and where the \textit{residual components} satisfy
\begin{align} &||r^k||\leq (1-\tau_1) ||Ax^k - b||, \label{eq:acc_r}\\
               &||s^k|| \leq (1 - \tau_2)||A^T y^k + z^k - c||. \label{eq:acc_s} \end{align}}
							\vspace{-0.5cm}
\STATE{Compute step-length $\alpha^{*,k} = \min\{\alpha^{*,k}_p,\alpha^{*,k}_d\}$ with 
\begin{align*} \alpha^{*,k}_p &= \max\{\alpha\in\mathbb{R}: x^k + \alpha \Delta x^k \geq 0\}, \\
               \alpha^{*,k}_d &= \max\{\alpha\in\mathbb{R}: z^k + \alpha \Delta z^k \geq 0\}. \end{align*} \label{alpha_star}}
							\vspace{-0.5cm}
\STATE{If $(x^k,y^k,z^k) + \alpha^{*,k}(\Delta x^k,\Delta y^k,\Delta z^k)$ is an  $(\epsilon,\epsilon_p,\epsilon_d)$-solution: \textbf{STOP}. \label{terminated}}
\STATE{Let $\bar{\alpha}^k$ be the largest value in $[0,1]$ such that for all $\alpha\in[0,\bar{\alpha}^k]:$
\begin{align*} &(x^k,y^k,z^k) + \alpha (\Delta x^k,\Delta y^k,\Delta z^k) \in\mathcal{N} \\
               &(x^{k} + \alpha \Delta x^k)^T (z^{k} + \alpha \Delta z^k) \leq ( 1-\alpha(1-\beta_2)) \left(x^k\right)^T z^k \end{align*} \label{alpha_bar}}
							\vspace{-0.5cm}
\STATE{Choose $\alpha_p^k,\alpha_d^k \in[0,1]$ such that $(x^{k+1},y^{k+1},z^{k+1}) \in \mathcal{N}$ and
\begin{equation}\label{eq:descent}
\left(x^{k+1}\right)^T z^{k+1} \leq (1 - \bar{\alpha}^k(1-\beta_3))\left(x^k\right)^T z^k 
\end{equation}
with $(x^{k+1},y^{k+1},z^{k+1})  = (x^{k} + \alpha_p^k\Delta x^{k},y^{k} + \alpha_d^k\Delta y^k ,z^{k} + \alpha_d^k \Delta z^k)$. \label{enough}}
\ENDFOR
\end{algorithmic}
\end{algorithm}

Using the definitions above, we can show the following lemma: 
\begin{lemma}[Upper bound on $\bar{\alpha}^k$ \cite{korzak2000convergence}] \label{thm:upper_bound_alpha}
If \cref{alg:KMM} does not stop on line \ref{terminated} of iteration $k$, then $\dot{\alpha}^{k}=\bar{\alpha}^k < \alpha^{*,k}$. 
\end{lemma}
\begin{proof}
The proof of this lemma is given in \cite{korzak2000convergence}. For the sake of completeness, we show the proof again using the notation of the current manuscript. 

 Suppose that $\dot{\alpha}^k \geq \alpha^{*,k}$. Then by definition of $\dot{\alpha}^k$ we have that $\alpha^{*,k}$ satisfies the conditions \cref{eq:ineq1,eq:ineq2,eq:ineq3}. However, by definition of $\alpha^{*,k}$ (see line \ref{alpha_star} in \cref{alg:KMM}) we know that there exists at least one index $i$ such that $(x_i^k + \alpha^{*,k}\Delta x_i^k)(z_i^k + \alpha^{*,k}\Delta z_i^k) = 0$. Combining this result with the fact that $f_i^k(\alpha^{*,k})\geq0$, we immediately get that $(x^k + \alpha^{*,k}\Delta x^k)^T (z^k + \alpha^{*,k}\Delta z^k) = 0$ for this particular index $i$. 

Now suppose $g_p^k(\alpha^{*,k}) < 0$, then we have $\|A(x^k + \alpha^{*,k}\Delta x^k) - b \| \leq \epsilon_p$. Otherwise, if $g_p^k(\alpha^{*,k}) \geq 0$, we get 
 $\|A(x^k + \alpha^{*,k}\Delta x^k) - b \| = 0 $ by definition of $g_p^k(\alpha)$ and the fact that $(x^k + \alpha^{*,k}\Delta x^k)^T (z^k + \alpha^{*,k}\Delta z^k) = 0$. Similarly, we can show the inequality $\|A^T(y^k + \alpha^{*,k}\Delta y^k)+  z^k + \alpha^{*,k}\Delta z^k - c \| \leq \epsilon_d$. This of course means that 
$(x^k,y^k,z^k) + \alpha^{*,k}(\Delta x^k,\Delta y^k,\Delta z^k)$ is an $(\epsilon,\epsilon_p,\epsilon_d)$-solution as defined by \cref{eq:def_conv}, which implies that the algorithm should have terminated on line \ref{terminated}. This contradicts the assumption made in the lemma and thus we have shown that $\dot{\alpha}^k < \alpha^{*,k}$. The proof now concludes by observing that the equality $\bar{\alpha}^k = \dot{\alpha}^k$ holds since $\dot{\alpha}^k$ satisfies the constraints $x^k + \dot{\alpha}^k\Delta x^k > 0$ and $z^k + \dot{\alpha}^k\Delta z^k > 0$ since we have a strict inequality $\dot{\alpha}^k < \alpha^{*,k}$.
\end{proof}

\begin{remark}\label{thm:compute_bar}
From this lemma it follows that we do not always need to compute $\bar{\alpha}^{k}$, since this result implies 
\begin{equation}\label{eq:ineq_use}
(1 - \alpha^{*,k}(1-\beta_3))(x^k)^Tz^k \leq (1 - \bar{\alpha}^k(1-\beta_3))(x^k)^Tz^k.
\end{equation}
Hence, we can compute trial step-lengths $\alpha_p^k = \beta_4 \alpha_p^{*,k}$ and $\alpha_d^k = \beta_4\alpha_d^{*,k}$ with parameter $\beta_4 \in (0,1)$ some constant close to one, e.g. $\beta_4 = 0.99995$, to make sure the iterates remain strictly positive. Then we check whether the new iterates $(x^{k+1},y^{k+1},z^{k+1})$ computed with these step-lengths remain in the neighborhood $\mathcal{N}$ and satisfy the descent condition $(x^{k + 1})^T z^{k + 1} \leq (1 - \alpha^{*,k}(1-\beta_3))(x^k)^Tz^k$, since this implies that \cref{eq:descent} is satisfied since we have the inequality \cref{eq:ineq_use}. 
\end{remark}

\begin{theorem}
\Cref{alg:KMM} terminates after a finite number of iterations for regularization parameters $\rho>0$ and $\delta>0$ sufficiently small. 
\end{theorem}
\begin{proof}
We follow the same structure as the proof in \cite{korzak2000convergence}, although the details are different since the regularization parameters $\rho$ and $\delta$ have to be carefully taken into account. 

We prove this claim by contradiction, so suppose that the algorithm does not terminate. Then it holds for all $k\geq 0$ that
\begin{equation}\label{eq:lower_bound_mu}
(x^k)^T z^k \geq \epsilon^* := \min\{\epsilon,\epsilon_p\gamma_p,\epsilon_d\gamma_d\}, 
\end{equation}
and that $||(x^k,z^k)||_1 \leq \omega$. This implies that the iterates $x^k$ and $z^k$ remain bounded and that there exists some constant $c_1$ such that $(x^k)^Tz^k \leq c_1$. This observation together with the fact that the iterates all satisfy $(x^k,y^k,z^k)\in\mathcal{N}$ allows us to show that the right-hand side in \cref{eq:update_eq} is bounded. For the first component we have:
\begin{align*}
\|b - Ax^k + r^k\| \leq \|b - Ax^k\| + \|r^k\| &\leq (2-\tau_1)\|Ax^k - b\| \\
&\leq (2-\tau_1) \max\left\{\frac{(x^k)^T z^k}{\gamma_p}, \epsilon_p \right\} \\ &\leq (2-\tau_1)\max\{c_1/\gamma_p,\epsilon_p\}.
\end{align*}
Similarly, we have for the second component in \cref{eq:update_eq} that
\begin{equation*}
\| c - A^T y^k - z^k \| \leq (2-\tau_2) \max\{c_1/\gamma_d,\epsilon_d\}. 
\end{equation*}
For the third component in \cref{eq:update_eq} we have:
\begin{equation*}
\| \beta_1 (x^k)^T z^k e / n - X^k z^k \| \leq \beta_1 (x^k)^T z^k \|e\|/n + \|X^k z^k \|\leq \beta_1 c_1 /\sqrt{n} + \|x^k\| \,\|z^k\|, 
\end{equation*}
which is also bounded since the iterates $x^k$ and $z^k$ are bounded. 

From Theorem 1 in \cite{armand2013uniform} and the fact that $(x^k)^T z^k \geq \epsilon^*$ we get that the matrices in \cref{eq:update_eq} have a bounded inverse, where the bound is independent of $k$. Hence, together with the boundedness of the right-hand side in \cref{eq:update_eq} we can conclude the $(\Delta x^k,\Delta y^k,\Delta z^k)$ is bounded. This implies that we can always find some constant $c_2 > 0$ such that
\begin{equation}\label{eq:bounded}
| \Delta x_i^k \Delta z_i^k - \gamma(\Delta x^k)^T\Delta z ^k/n| \leq c_2 \text{ and } |(\Delta x^k)^T \Delta z^k| \leq c_2. 
\end{equation}
We will now show that we can use this result to prove that $\bar{\alpha}^k$ as defined on line \ref{alpha_bar} in \cref{alg:KMM} is bounded below by some value $\alpha^*>0$. 

For all $k\geq 0$ and $i = 1,\ldots,n$ we have from the third component in \cref{eq:update_eq} that
\begin{equation}\label{eq:third_component}
z_i^k \Delta x_i ^k + x_i^k \Delta z_i ^k = \beta_1 (x^k)^T z^k /n - x_i^k z_i^k.
\end{equation}
From this it follows that
\begin{equation} \label{eq:xz_i}
(x_i^k + \alpha \Delta x^k_i)(z_i^k + \alpha \Delta z^k_i) = (1- \alpha) x_i^k z_i^k + \alpha \beta_1 (x^k)^T z^k /n + \alpha^2 \Delta x_i^k \Delta z_i^k.
\end{equation}
By taking the sum for $i=1,\ldots,n$ of both sides of equality \cref{eq:third_component} we get
\begin{equation*}
(z^k)^T \Delta x^k + (x^k)^T \Delta z^k = (\beta_1 - 1) (x^k)^T z^k
\end{equation*}
from which we get
\begin{equation}\label{eq:xtz}
(x^k + \alpha\Delta x^k)^T(z^k + \alpha\Delta z^k) = (1 +\alpha(\beta_1 - 1)) (x^k)^T z^k + \alpha^2 (\Delta x^k)^T \Delta z^k. 
\end{equation}
Substituting \cref{eq:xz_i,eq:xtz} in the definition of $f_i^k(\alpha)$ in \cref{eq:def_f} we get: 
\begin{align*}
f_i^k(\alpha)  &= (x_i^k + \alpha \Delta x^k_i)(z_i^k + \alpha \Delta z^k_i) - \gamma (x^k + \alpha\Delta x^k)^T(z^k + \alpha\Delta z^k)/n, \\
&= (1- \alpha) x_i^k z_i^k + \alpha \beta_1 (x^k)^T z^k /n +\alpha^2 \Delta x_i^k \Delta z_i^k \\ 
&\hspace{4cm} -\gamma(1 +\alpha(\beta_1 - 1)) (x^k)^T z^k/n - \gamma \alpha^2 (\Delta x^k)^T \Delta z^k / n \\
&= \underbrace{(1 -\alpha)\left(x_i^k z_i^k - \gamma (x^k)^T z^k/n \right)}_{\geq0} + \alpha^2 \underbrace{\left(\Delta x_i^k \Delta z_i^k - \gamma(\Delta x^k)^T\Delta z ^k/n\right)}_{\geq - c_2} \\
&\hspace{6.5cm}+ \alpha\beta_1(1-\gamma)\underbrace{(x^k)^T z^k}_{\geq \epsilon^*}/n \\
&\geq -c_2 \alpha^2 + \alpha\beta_1(1-\gamma)\epsilon^*/n = \alpha( \beta_1(1-\gamma)\epsilon^*/n - c_2 \alpha).
\end{align*}
Hence it follows that for all $\alpha\leq\frac{\beta_1(1-\gamma)\epsilon^*}{c_2 n}$ we have $f_{i}(\alpha)\geq0$.

Let us now consider the function $g_p^k(\alpha)$ defined in \cref{eq:def_g_primal}. From the first component in \cref{eq:update_eq} we have
\begin{align*}
A(x^k + \alpha \Delta x^k) - b &= Ax^k - b + \alpha A\Delta x^k \\
&= Ax^k - b + \alpha(b - Ax^k + r^k - \delta \Delta y^k) \\
&= (1-\alpha)(Ax^k - b) + \alpha r^k - \alpha\delta\Delta y^k.
 \end{align*}
By taking norms and using \cref{eq:acc_r} we get
\begin{equation} \label{eq:bound_primal}
\|A(x^k + \alpha \Delta x^k) - b\| \leq ( 1 -\alpha \tau_1) \|Ax^k - b\| + \alpha \delta \| \Delta y^k\|.
\end{equation}
Suppose first that $g_p^k(0) <0$, then we have that $\| Ax^k - b\| \leq \epsilon_p$ since $(x^k,y^k,z^k)\in\mathcal{N}$. 
Now because of the boundedness of $\Delta y^k$ there exists a constant $c_3>0$ such that $\|\Delta y^k\| \leq c_3$ for all $k$. Now if we choose $\delta>0$ small enough such that $\delta \leq \tau_1 \epsilon_p/c_3$ then we have
\begin{align*}
\|A(x^k + \alpha \Delta x^k) - b\| & \leq (1-\alpha\tau_1)\epsilon_p + \alpha\tau_1\epsilon_p\frac{\|\Delta y^k\|}{c_3} \leq(1-\alpha\tau_1)\epsilon_p + \alpha\tau_1\epsilon_p = \epsilon_p.
\end{align*}
Now let us consider the case $g_p^k(0)\geq0$. Using \cref{eq:bounded,eq:xtz,eq:bound_primal} and the definition \cref{eq:def_g_primal} we obtain
\begin{align*}
g_p^k(\alpha) &= (x^k + \alpha\Delta x^k)^T(z^k + \alpha\Delta z^k) - \gamma_p \|A(x^k + \alpha\Delta x^k) - b \| \\
&=  (1 +\alpha(\beta_1 - 1)) (x^k)^T z^k + \alpha^2 (\Delta x^k)^T \Delta z^k - \gamma_p \|A(x^k + \alpha\Delta x^k) - b\| \\
&=  \alpha(\beta_1 + \tau_1 - 1) \underbrace{(x^k)^T z^k}_{\geq \epsilon^*} + (1 - \alpha \tau_1) (x^k)^T z^k + \alpha^2 \underbrace{(\Delta x^k)^T \Delta z^k}_{\geq -c_2}    \\
&\hspace{8cm} - \gamma_p \|A(x^k + \alpha\Delta x^k) - b\| \\
&\geq \alpha(\beta_1 + \tau_1 - 1)\epsilon^* + (1 - \alpha \tau_1) (x^k)^T z^k - c_2\alpha^2 \\
&\hspace{5cm} - \gamma_p( 1 -\alpha \tau_1) \|Ax^k - b\| - \gamma_p \alpha \delta \| \Delta y^k\| \\
& =  \alpha(\beta_1 + \tau_1 - 1)\epsilon^*  + (1 - \alpha \tau_1) \underbrace{\left( (x^k)^T z^k - \gamma_p\|Ax^k - b\|   \right)}_{\geq0} - c_2\alpha^2 -\gamma_p \alpha \delta \| \Delta y^k\| \\
&\geq  \alpha\left( (\beta_1 + \tau_1 - 1)\epsilon^* - \gamma_p \delta \| \Delta y^k\| - c_2 \alpha \right).
\end{align*}
Now choose $\delta$ small enough such that $\delta \leq \frac{(\beta_1 + \tau_1 - 1)\epsilon^*}{2\gamma_p c_3}$, then we have
\begin{align*}
(\beta_1 + \tau_1 - 1)\epsilon^* - \gamma_p \delta \| \Delta y^k\| &\geq (\beta_1 + \tau_1 - 1)\epsilon^* - \frac{(\beta_1 + \tau_1 - 1)\epsilon^*\|\Delta y^k \|}{2c_3} \\ &\geq (\beta_1 + \tau_1 - 1)\epsilon^*/2. 
\end{align*}
Now it easily follows that $g_p^k(\alpha)\geq 0$ for all $\alpha \leq \frac{(\beta_1 + \tau_1 - 1)\epsilon^*}{2c_2}$.

We can treat $g_d^k(\alpha)$ in a similar fashion. We leave out the details, but for completeness we summarize the main important results. First of all, by the second component of \cref{eq:update_eq} we have
\begin{equation*}
A^T(y^k + \alpha\Delta y^k) + z^k + \alpha \Delta z^k - c = (1-\alpha)(A^Ty^k + z^k -c) + \alpha s^k + \alpha\rho \Delta x^k
\end{equation*}
from which we get by using \cref{eq:acc_s}
\begin{equation}\label{eq:bound_dual}
\|A^T(y^k + \alpha\Delta y^k) + z^k + \alpha \Delta z^k - c \| \leq (1-\alpha\tau_2) \| A^Ty^k + z^k -c\| + \alpha\rho \|\Delta x^k\|.
\end{equation}
Since $\Delta x^k$ is bounded we have some constant $c_4>0$ such that $\|\Delta x^k\| \leq c_4$ for all $k$.
Suppose first that $g_d^k(0)<0$, which implies $\| A^Ty^k + z^k -c\|\leq\epsilon_d$, and take $\rho \leq \tau_2 \epsilon_d/c_4$, then we have 
\begin{equation*}
\|A^T(y^k + \alpha\Delta y^k) + z^k + \alpha \Delta z^k - c \| \leq (1-\alpha\tau_2)\epsilon_d + \alpha\rho \|\Delta x_k\| \leq \epsilon_p.
\end{equation*}
For the case $g_d^k(0)\geq0$ we have, following the same steps as for $g_p^k(\alpha)$, using \cref{eq:bounded,eq:xtz,eq:bound_dual} and the definition \cref{eq:def_g_dual}:
\begin{equation*}
g_d^k(\alpha)\geq  \alpha\left( (\beta_1 + \tau_2 - 1)\epsilon^* - \gamma_d \rho \| \Delta x^k\| - c_2 \alpha \right).
\end{equation*}
If we now have chosen $\rho$ small enough such that $\rho \leq \frac{(\beta_1 + \tau_2 - 1)\epsilon^*}{2\gamma_d c_4}$ then $g_d^k(\alpha)\geq0$ holds for all $\alpha \leq \frac{(\beta_1 + \tau_2 - 1)\epsilon^*}{2c_2}$.

Now the only thing we still need to consider is the inequality $h^k(\alpha)\geq0$ defined by \cref{eq:def_h}. Again using \cref{eq:xtz} we have
\begin{align*}
h^k(\alpha) &= (1-\alpha(1 - \beta_2)) \left(x^{k}\right)^T z^k -(x^k + \alpha\Delta x^k)^T(z^k + \alpha\Delta z^k) \\
&=  (1-\alpha(1 - \beta_2)) \left(x^{k}\right)^T z^k -  (1 +\alpha(\beta_1 - 1)) (x^k)^T z^k - \alpha^2 (\Delta x^k)^T \Delta z^k \\
&= \alpha(\beta_2 - \beta_1) (x^k)^T z^k -  \alpha^2 (\Delta x^k)^T \Delta z^k \\
&\geq \alpha((\beta_2 - \beta_1)\epsilon^* - \alpha c_2).
\end{align*} 
Hence, it follows that $h^k(\alpha)\geq 0$ for all $\alpha\leq(\beta_2 - \beta_1)\epsilon^*/c_2.$ 

Now using the results above and \cref{thm:upper_bound_alpha}, we get 
\begin{equation*}
\bar{\alpha}^k \geq \alpha^* := \min\left\{1, \frac{\beta_1(1-\gamma)\epsilon^*}{c_2 n}, \frac{(\beta_1 + \tau_1 - 1)\epsilon^*}{2c_2},\frac{(\beta_1 + \tau_2 - 1)\epsilon^*}{2c_2}, \frac{(\beta_2 - \beta_1)\epsilon^*}{c_2}\right\}
\end{equation*}
if $\rho$ and $\delta$ are chosen sufficiently small, more specifically if they satisfy 
\begin{equation*}
\delta \leq \min \left\{ \frac{\tau_1\epsilon_p}{c_3}, \frac{(\beta_1 + \tau_1 - 1)\epsilon^*}{2\gamma_p c_3}\right\} \text{ and }  \rho \leq \min \left\{ \frac{\tau_2\epsilon_d}{c_4}, \frac{(\beta_1 + \tau_2 - 1)\epsilon^*}{2\gamma_d c_4}\right\}.
\end{equation*}
This then leads to $(x^k)^T z^k \leq (1-\alpha^*(1-\beta_3))^k (x^0)^T z^0$, which implies that we have $\lim_{k\rightarrow \infty} (x^k)^T z^k = 0$. This contradicts \cref{eq:lower_bound_mu}, which concludes the proof.
\end{proof}

\begin{remark}\label{thm:omega_stop}
Let us conclude this section by briefly explaining the case when \cref{alg:KMM} terminates with some iterate satisfying $\|(x^k,z^k)\|_1>\omega$. The authors in \cite{kojima1993primal} show that in the exact and unregularized case (i.e. $\tau_1 = \tau_2 = 1$ and $\rho = \delta = 0$)  this implies there exists no feasible point of the primal-dual pair of linear programming problems \cref{eq:primal,eq:dual} in a large subspace of $\mathbb{R}^{n+m+n}$. The larger we take $\omega$ the larger this subspace of infeasible points becomes. This means that the primal-dual pair \cref{eq:primal,eq:dual} is likely to be infeasible, although no guarantee is given. As was pointed out by the authors in \cite{korzak2000convergence}, it does not seem possible to prove this result in the inexact case. Hence, do not attempt to prove this result for the inexact and regularized case. However, as they point out, this might be of little practical importance since the stopping criterion $\|(x^k,z^k)\|>\omega$ will in most cases not be satisfied in a reasonable amount of time. This downside of not being able to definitely determine infeasibility is thus inherited by \cref{alg:KMM}.
\end{remark}

\section{Implementation details} \label{sec:implementation}

In this section we explain some details of our implementation used in the numerical experiments of \cref{sec:num_ex}. We start with a general discussion on the reformulation of the linear system of equations, which is an important aspect for all primal-dual interior-point methods. Next, we explain how we can leverage single precision floating point arithmetic when solving linear programming problems with a dense constraint matrix. 

\subsection{Reformulation of the linear system of equations} \label{sec:reformulations}
Matrices of the form shown in \cref{eq:newton_eq} or \cref{eq:newton_eq_reg} are always transformed in computational practice to either a symmetric indefinite matrix or a positive definite matrix. Let us briefly explain how this can be achieved by considering the linear system \cref{eq:newton_eq_reg} for a general right-hand side $(\xi,\zeta,\eta) \in\mathbb{R}^{m + n + n}$. Writing out the equations we get: 
\begin{align}
&A\Delta x + \delta \Delta y = \xi, \label{eq:eq1} \\
 &A^T \Delta y + \Delta z - \rho\Delta x  = \zeta, \label{eq:eq2}\\ 
&Z\Delta x + X \Delta z = \eta. \label{eq:eq3}
\end{align}
Isolating $\Delta z$ from \cref{eq:eq3} and combining this with \cref{eq:eq2} we get
\begin{align}
&A^T \Delta y + X^{-1}(\eta - Z\Delta x) - \rho \Delta x = \zeta \nonumber \\
\Leftrightarrow &A^T \Delta y - \left(X^{-1}Z + \rho I\right)\Delta x = \zeta - X^{-1} \eta. \label{eq:dx}
\end{align} 
This leads to the so-called augmented system approach with a symmetric indefinite matrix
\begin{equation}\label{eq:augmented}
\begin{pmatrix} - D & A^T \\ A & \delta I \end{pmatrix}\begin{pmatrix}\Delta x\\ \Delta y\end{pmatrix} = \begin{pmatrix}\zeta - X^{-1} \eta  \\ \xi \end{pmatrix},
\end{equation}
where we introduce the notation $D = (X^{-1}Z + \rho I)$. After (approximately) solving for $\Delta x$ and $\Delta y$ we can recover $\Delta z$ from \cref{eq:eq3}. Even if \cref{eq:eq1,eq:eq2} are not satisfied exactly, for instance if we use a Krylov subspace method to obtain an approximate solution to \cref{eq:augmented}, we can always make sure that \cref{eq:eq3} does hold exactly (assuming exact arithmetic). Hence, we can assume a zero block in the final component of the residual in \cref{eq:update_eq}. 
\begin{remark}
The matrix in \cref{eq:augmented} is symmetric quasidefinite for $\rho$ and $\delta$ strictly positive \cite{vanderbei1995symmetric}. This implies that it allows a Cholesky-type factorization of the form $L\Delta L^T$ where $L$ is a lower triangular matrix and $\Delta$ is diagonal matrix containing both positive and negative entries, i.e. there is no need for $2\times 2$ pivoting. This also means that we can construct a permutation matrix solely to preserve as much sparsity as possible in the Cholesky-like factors. This permutation matrix need not change over the different interior-point iterations since the sparsity pattern of the augmented matrix stays the same. The sparse Cholesky-factorization is generally more efficient than indefinite solvers. These observations have for instance been exploited in \cite{vanderbei1995symmetric,saunders1996solving}. For more information on the stability of the Cholesky factorization applied to symmetric quasidefinite matrices we refer to \cite{gill1996stability}.
\end{remark}
We can reduce the equations in \cref{eq:augmented} a bit further by elimination of $\Delta x$. From \cref{eq:dx} we get
\begin{equation} \label{eq:dx_anders}
\Delta x = D^{-1} \left(A^T \Delta y - \zeta  + X^{-1} \eta \right).
\end{equation}
Multiplying both sides of this equation with $A$ and using the fact that $A\Delta x = \xi - \delta \Delta y$ from \cref{eq:eq1}, we get after reordering the terms:
\begin{equation}\label{eq:normal_eq}
 \left(A D^{-1} A^T + \delta I\right) \Delta y = \xi  + AD^{-1}(\zeta - X^{-1}\eta).
\end{equation}
This is known as the normal equations approach. The matrix here is positive definite for $\delta>0$ (even when $A$ is rank deficient), so we can for instance use a Cholesky factorization to solve this linear system of equations. After (approximately) solving for $\Delta y$ we can compute $\Delta x$ using \cref{eq:dx_anders} and $\Delta z$ again using \cref{eq:eq3}. Then we can show (again assuming exact arithmetic) that \cref{eq:eq2} also holds exactly, even if we do not solve \cref{eq:normal_eq} very accurately. Indeed, from \cref{eq:dx_anders} we get
\begin{equation*}
(X^{-1}Z+ \rho I) \Delta x = A^T \Delta y - \zeta  + X^{-1} \eta \Leftrightarrow  A^T \Delta y + \underbrace{X^{-1}(\eta - Z \Delta x)}_{= \Delta z \text{ using } \cref{eq:eq3}} - \rho \Delta x = \zeta.
\end{equation*}
Alternatively we can also compute $\Delta z$ directly from \cref{eq:eq2} and show that \cref{eq:eq3} also holds exactly in this case, even if \cref{eq:normal_eq} is not solved very accurately. Indeed, again starting from \cref{eq:dx_anders} we have
\begin{equation}
X^{-1}Z \Delta x + \underbrace{\rho \Delta x - A^T \Delta y + \zeta}_{= \Delta z \text{ using } \cref{eq:eq2}} = X^{-1}\eta.
\end{equation}
Now by multiplying both sides of this equation with $X$ we see that \cref{eq:eq3} holds exactly. This means that we can assume that in both these cases we have $\tau_2 = 1$ in \cref{eq:update_eq} if we use the normal equations approach. From computational experience we believe that this second option, i.e. computing $\Delta z$ using $\cref{eq:eq2}$, is slightly more numerically stable. 

Lastly we mention that we can monitor the norm of the residual term $r^k$ in \cref{eq:update_eq}, i.e. the inexactness in \cref{eq:eq1}, by computing the residual of the linear system \cref{eq:normal_eq}. This is because -- again using \cref{eq:dx_anders} --  we have
\begin{align}
\| A\Delta x + \delta \Delta y - \xi \| &= \| A  D^{-1} \left(A^T \Delta y - \zeta  + X^{-1} \eta \right) + \delta \Delta y - \xi \| \nonumber \\
&= \| \left(A D^{-1} A^T + \delta I\right) \Delta y - \xi  - AD^{-1}(\zeta - X^{-1}\eta) \|. \label{eq:stopping}
\end{align}
This is for instance useful when we use a Krylov subspace method, such as the Conjugate Gradient method, to solve the linear system \cref{eq:normal_eq} up to a certain accuracy. We can simply state the stopping criterion of this iterative method to make sure that \cref{eq:acc_r} is satisfied. This is in fact the approach that we consider in our experiments.

\subsection{Mixed-precision linear equation solver}

As previously mentioned we can simply use any specialized linear algebra routine developed in other work on inexact interior-point methods and incorporate these techniques in \cref{alg:KMM}. The focus in literature has mostly been on linear programming problems with a sparse constraint matrix, where the linear systems are solved using some Krylov subspace method with a specialized preconditioner \cite{gondzio2013convergence,cui2019implementation,monteiro2003convergence,al2009convergence}. However, for the remainder of this work we consider linear programming problems with a dense constraint matrix. These types of problems arise for instance in Basis Persuit denoising \cite{chen2001atomic}.

We aim to exploit the increased efficiency of IEEE single precision compared to IEEE double precision \cite{4610935,30711}. The double precision floating-point format is comprised of 64 bits, while single precision only uses 32 bits. In general we can execute single precision floating point operations (\textsc{flops}) approximately twice as fast as double precision \textsc{flops}. However, operations using single precision are also less accurate since the unit roundoff of single precision is $u_s = 2^{-24} \approx 6.0 \times 10^{-8}$, while double precision has a unit roundoff $u_d = 2^{-53} \approx 1.1 \times 10^{-16}$. IEEE half precision, which is a floating-point format using only 16 bits with unit roundoff $u_h = 2^{-11} \approx 4.9 \times 10^{-4}$, has also recently gained a lot of popularity due to a growing availability in hardware support (for instance on the NVIDIA P100 and V100 GPUs and the AMD Radeon Instinct MI25 GPU \cite{higham2019squeezing}). The unit roundoff of the precision is determined by the number of bits in the `significand', while the number of `exponent' bits determines the range of numeric values, see \cref{fig:bits} for an illustration. Switching to a lower precision thus not only decreases the accuracy, but also the range of values. This becomes especially important when considering half precision which has a limited range of $\pm 65,500$.

\begin{figure}
	\centering
	\includegraphics[width=1\linewidth]{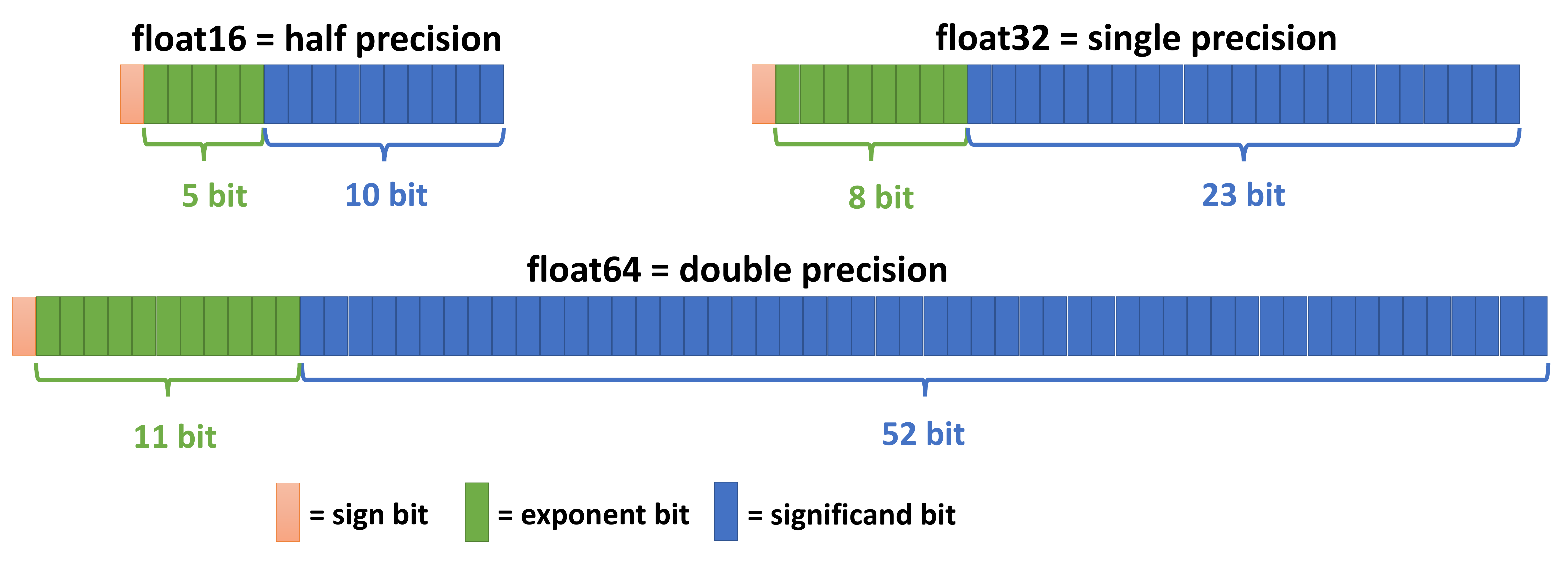} 
	\caption{Illustration of the IEEE half, single and double precision floating point formats. \label{fig:bits}}
\end{figure}

Operations in half precision can generally be executed twice as fast as single precision operations, and even up to 8 times faster on specialized tensor cores (for instance on the NVIDIA V100). However, in this paper we focus our attention on exploiting single precision arithmetic, since for the moment we have no access to hardware that supports half-precision arithmetic. In contrast, single precision arithmetic is available on almost all off-the-shelf hardware. Half precision floating point operations can of course easily be simulated in software, but this will evidently not lead to any performance gains compared to single precision. Hence, we leave a detailed discussion of half-precision arithmetic as possible future work.

Let us now turn our attention to how we can exploit single precision when solving a linear system of equations $M v = w$ for some general positive definite matrix $M\in\mathbb{R}^{m\times m}$. After performing a Cholesky factorization of $M$ in exact arithmetic we obtain an exact equality $M =L L^T$ with $L$ a lower triangular matrix. Solving a linear system $M v = w$ can then be done very efficiently by first solving the lower-triangular system $Ly = w$ for $y$ and subsequently solving the upper-triangular system $L^T v = y$ for $v$. Computing the Cholesky factorization requires approximately $m^3/3$ floating point operations, while solving a triangular linear system using forward or backward substitution requires $m^2$ \textsc{flops}. All these operations are of course subject to rounding errors in finite precision arithmetic.  An important consequence is that the Choleksy factorization might break down when the matrix $M$ is not sufficiently positive definite with respect to the floating point precision used for the factorization. However, even in such a case when the matrix is numerically indefinite or singular, we can still obtain useful information from the Choleksy factorization.

In \cite{higham2021exploiting,higham2019squeezing} the authors consider a diagonal perturbation of the matrix $M$ and perform a Cholesky factorization of the shifted matrix $M + c_m u\diag(M)$ instead. Here, $c_m$ is some constant depending on the dimension of the problem and $u$ is the unit roundoff of the precision used for the factorization. By a result from Demmel \cite{demmel1989floating} it follows that the Cholesky factorization of this perturbed matrix will always succeed for $c_m$ sufficiently large. 

For an $LU$ factorization in finite precision arithmetic it has been reported in literature that the computed $L$ and $U$ factors still contain useful information in the case when the condition number $\kappa(M)$ is (much) larger than $1/u$ and thus numerically singular.  More specifically it has been observed that the approximation $\kappa(L^{-1}MU^{-1}) \approx 1 + \kappa(M)u$ often holds in practice \cite{carson2017new}. We illustrate by a small example that this approximation also holds when computing a Cholesky factorization $LL^T \approx M + c_m u\diag(M)$ for some precision $u$. We generate a sequence of matrices $M\in\mathbb{R}^{200\times 200}$ with smallest eigenvalue $\lambda_{\min}(M) = 1$ and largest eigenvalue $\lambda_{\max}(M)=10^k$ for $k$ ranging from $0.1$ to $16$. Then we perform a Choleksy factorization in single precision of the matrix $M + 10u_s\diag(M)$, where $u_s$ is the unit roundoff, and then show the (2-norm) condition number $\kappa(M) = 10^k$ , $\kappa(L^{-1}ML^{-T})$ and the approximation $\kappa(L^{-1}ML^{-T}) \approx 1 + \kappa(M)u$ for the computed factor $L$. The result is given by \cref{fig:cond_est} (left). The same procedure is repeated for half precision. In the latter case we should also be careful with the possibility of overflow and underflow due to the limited range of half-precision floating point numbers. Hence, we first perform a diagonal scaling $H = D_M^{-1} M D_M^{-1}$ with $D_M = \sqrt{\diag(M)}$ as suggested in \cite{higham2021exploiting,higham2019squeezing}, which causes all diagonal entries to be one and all non-diagonal entries to be smaller than one (in absolute value). A Cholesky factorization is then computed in half precision of the perturbed matrix $H +10 u_h I$. The condition numbers $\kappa(M)$, $\kappa(L^{-1}HL^{-T})$ and approximation $\kappa(L^{-1}HL^{-T}) \approx 1 + \kappa(H) u_s$ are shown in \cref{fig:cond_est} (right). 

\begin{figure}
	\centering
	\begin{tabular}{cc}
	\includegraphics[width=0.5\linewidth]{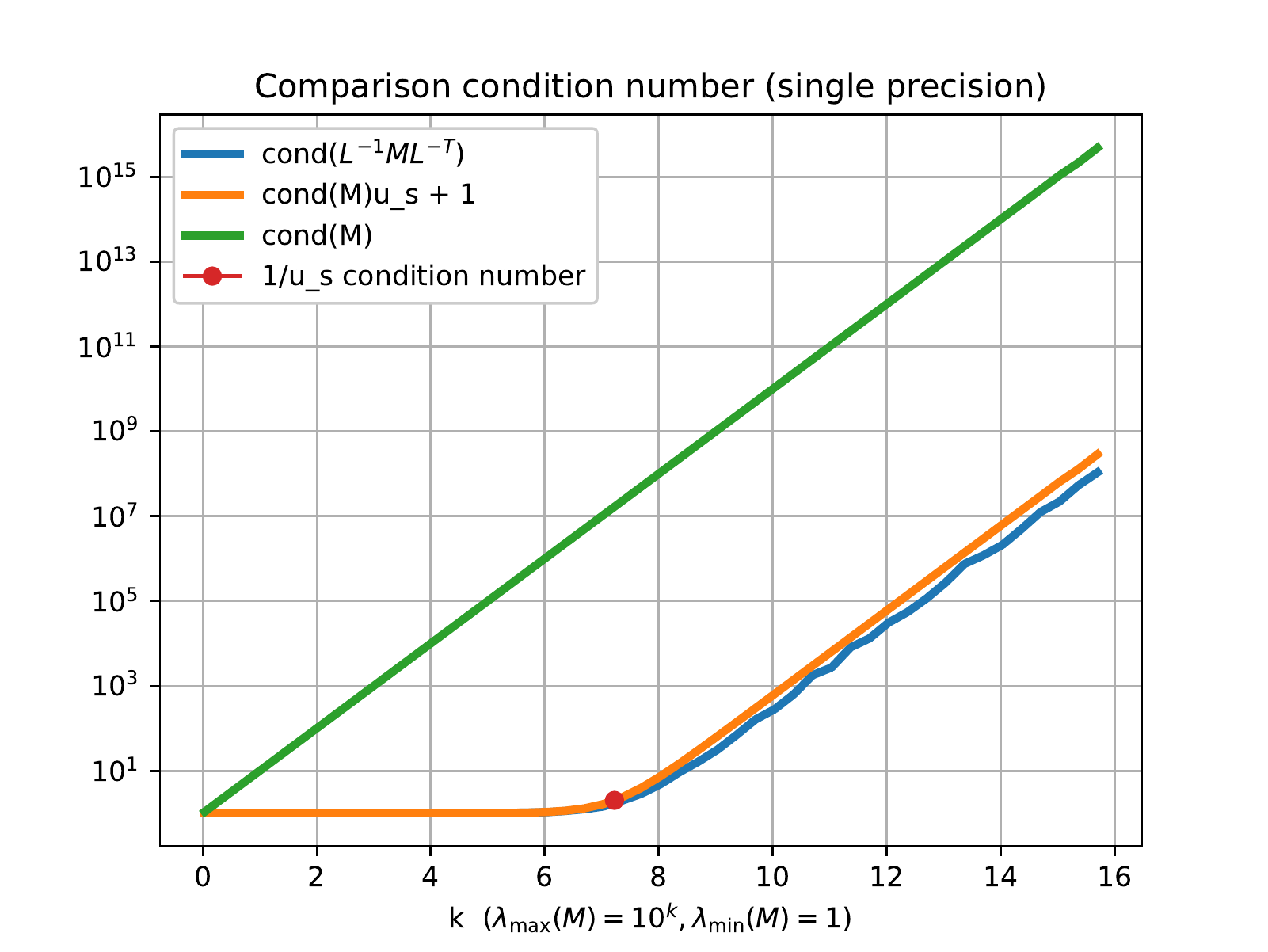} & \hspace{-1cm} \includegraphics[width=0.5\linewidth]{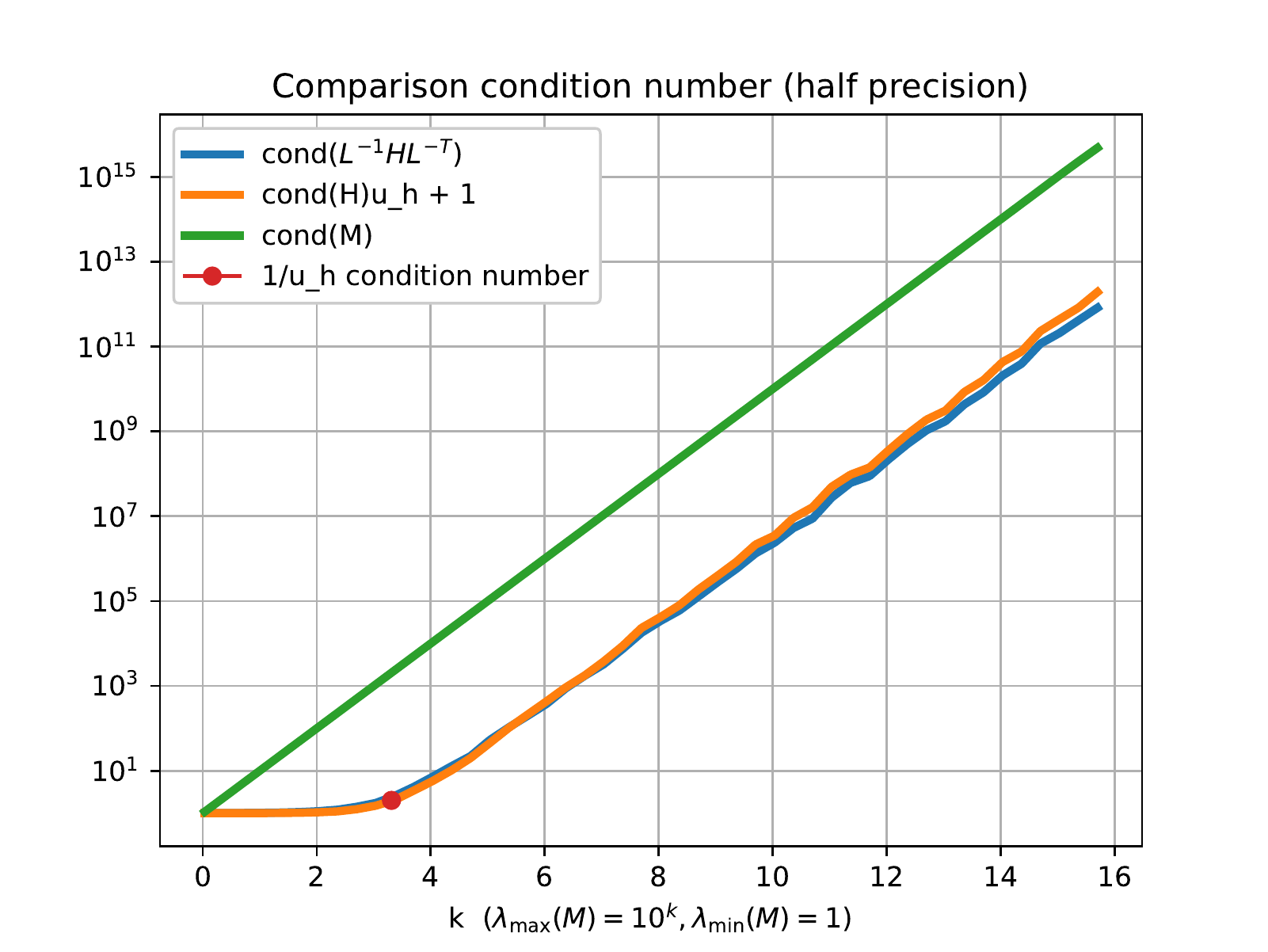}\\
	\end{tabular}
	\caption{Condition numbers $\kappa(M)$ and $\kappa(L^{-1}ML^{-T})$, together with the approximation for the condition number of the preconditioned matrix: $\kappa(M)u + 1$. Left: unit roundoff $u=u_s$ for single precision, right: $u=u_h$ for half precision. \label{fig:cond_est}}
\end{figure}

We can observe that if $\kappa(A)\leq 1/u$ then we obtain a good Cholesky factor, in the sence that $\kappa(L^{-1}A L^{-T})$ is close to one. However, even when the matrix is numerically singular (with respect to the precision of the factorization), we still have a significant reduction, proportional to the unit roundoff $u$, in the condition number. 
This implies that, even when the Cholesky factor can not be used to compute a very accurate solution of the linear system of equations, we can still use it as a preconditioner in the Conjugate Gradient method. Since convergence of the Conjugate Gradient method is determined by the condition number of the preconditioned matrix $L^{-1}ML^{-T}$, we can expect very rapid convergence as long as $M$ is not too ill-conditioned. 

Hence, to summarize we can solve a linear system of equations $M v = w$ using a mixed-precision approach, where we first compute a Cholesky factor $L$ of the shifted matrix $M + c_m u_s \diag(M)$ in single precision and apply the Conjugate Gradient method implemented in double precision. The preconditioner $L^{-T}L^{-1}$ is applied in single precision using backward and forward substitution, with the final result converted back to double precision. Note that in this case the iterations of the Conjugate Gradient method can be interpreted as some kind of iterative refinement. A closely related but slightly different approach was explored in \cite{higham2021exploiting}.

\begin{algorithm}
\caption{Mixed-precision solver for \cref{eq:normal_eq}. \hfill }
\label{alg:solver}
\begin{algorithmic}[1]
\STATE{\textbf{Input:}$A,x,z,\xi,\zeta,\eta,c_m,\rho,\delta,\text{tol}$}
\STATE{Compute diagonal matrix $D = X^{-1}Z + \rho I$  in \texttt{float64}} 
\STATE{Compute right-hand side $w = \xi + AD^{-1}\left(\zeta - X^{-1}\eta\right)$ in \texttt{float64}}
\STATE{Compute diagonal matrix $\tilde{D} = \sqrt{D^{-1}}$ in \texttt{float64}}
\STATE{Compute $\tilde{M}= A\tilde{D}$ in \texttt{float32}. \label{root}}
\STATE{Compute $M = \tilde{M}\tilde{M}^T + \delta I$ in \texttt{float32}.}
\STATE{Apply Cholesky to $M + c_m\diag(M)$ in \texttt{float32} and store factor $L$ in \texttt{float32}.}
\STATE{Apply the (preconditioned) Conjugate Gradient method in \texttt{float64} to 
\begin{equation*}
L^{-T}L^{-1} (AD^{-1}A + \delta I)\Delta y = L^{-T}L^{-1} w.
\end{equation*} 
Matrix vector products with $(AD^{-1}A + \delta I)$ are computed in \texttt{float64} without explicitly forming the matrix and the preconditioner $L^{-T}L^{-1}$ is applied using backward and forward substitution in \texttt{float32} with the result converted to \texttt{float64}. Conjugate Gradients is terminated when $\|(AD^{-1}A + \delta I)\Delta y - w \| \leq \text{tol}$.}
\STATE{\textbf{Output:} Approximate solution $\Delta y$ to the linear system \cref{eq:normal_eq}.}
\end{algorithmic}
\end{algorithm}

This approach can obviously be applied to the linear system of equations \cref{eq:normal_eq}, since the matrix is positive definite for $\delta>0$, even when $A$ is rank deficient. However, an important point that should not be overlooked is the fact that a dense matrix-matrix multiplication is also quite expensive. Explicitly forming the normal equations $AD^{-1}A^T$ is in fact more expensive than the Cholesky factorization in terms of number of operations, since the former operation requires approximately $m^2n$ \textsc{flops} (assuming we exploit the symmetry), while the latter requires about $m^3/3$ \textsc{flops}. Hence, forming the normal equations should also be done in single precision. We summarize our approach to solve \cref{eq:normal_eq} in \cref{alg:solver}, where we denote double precision and single precision with \texttt{float64} and \texttt{float32} respectively.

We can apply the mixed-precision solver \cref{alg:solver} in the context of the regularized inexact interior-point method \cref{alg:KMM}. To obtain an \textit{inexact} search direction $(\Delta x_k,\Delta y_k,\Delta z_k)$ satisfying \cref{eq:update_eq} together with \cref{eq:acc_r,eq:acc_s} we simply apply \cref{alg:solver} with input parameters  $x = x_k,\, z = z_k,\, \xi = b - Ax^k,\, \zeta = c - A^T y^k - z^k,\, \eta = \beta_1 \mu_k e - X^k z^k$ and $\text{tol}=(1-\tau_1) \|Ax^k - b\|$. We then compute $\Delta x^k$ using \cref{eq:dx_anders} and $\Delta z^k$ using \cref{eq:eq2}. Assuming that the only source of inexactness is in solving the linear system \cref{eq:normal_eq} we have that \cref{eq:acc_r} is satisfied (since we have \cref{eq:stopping}) and \cref{eq:acc_s} is satisfied with $\tau_2 = 1$.

\section{Numerical experiments}\label{sec:num_ex}
All experiments are performed on a laptop with Intel(R) Core(TM) i7-7700HQ CPU @ 2.80GHz. 
We have implemented \cref{alg:KMM} with the details as explained in \cref{sec:implementation} in Python 3.8.3 using the Spyder IDE version 4.2.0. 
The implementation is based on NumPy\footnote{\url{https://numpy.org/}} which allows us to easily convert variables from double precision to single precision using $\texttt{numpy.single}(\cdot)$ and conversely from single to double using $\texttt{numpy.double}(\cdot)$. The Cholesky factorization is implemented using SciPy\footnote{\url{https://www.scipy.org/}} as $\texttt{L}=\texttt{scipy.linalg.cholesky(M,lower = True)}$. The preconditioner application $L^{-T}L^{-1} w$ can then be efficiently implemented as
\begin{equation*}
\texttt{numpy.double(scipy.linalg.cho}\_\texttt{solve((L,True),numpy.single(w)))}.
\end{equation*}

Unfortunately there do not seem to be many suitable benchmark linear programming problems of the form \cref{eq:primal} with a dense constraint matrix $A$. Since we only want to perform experiments that are reproducible we use (sparse) test-problems from the NETLIB LP Test problem set\footnote{\url{https://www.numerical.rl.ac.uk/cute/netlib.html}}, but treat these matrices as if they were dense. We consider a selection of $14$ linear programming problems in standard form, where the constraint matrix $A$ has dimensions $m\times n$ with $m\geq 1000$ and $n \leq 10,000$. This selection is based on the fact that the benefit of single precision is expected to be much less apparent when the computational work is limited. 

We now list our choices of parameters in \cref{alg:KMM} which apply for all experiments, unless explicitly states otherwise. For the stopping criterion of the interior-point method we use $\omega = 10^{40}$ and tolerances $\epsilon = 10^{-8}n$ for the complementary condition and $\epsilon_p= 10^{-8}\|b\|$ and $\epsilon_d  = 10^{-8}\|c\|$ for the primal and dual feasibility respectively. Next, we use centering parameter $\beta_1 = 0.1$ and parameters $\beta_2=0.9$ and $\beta_3=0.95$ for the descent condition. For the conditions in \cref{eq:acc_r,eq:acc_s} we set $\tau_1 = 0.95$ and $\tau_2=1$ respectively. Regularization parameters $\rho =\delta=10^{-10}$ were sufficiently small for all test-problems considered. The constant $c_m = 30$ is chosen for the mixed-precision solver \cref{alg:solver}. Lastly we use neighborhood parameters $\gamma = \gamma_p = \gamma_d = 10^{-8}$ in \cref{eq:N}. The initial point $(x_0,y_0,z_0)$ is computed as explained in section 14.2 of \cite{nocedal2006numerical} with a small modification to handle rank deficient matrices $A$. 

Before calculating $\bar{\alpha}^k$ on line \ref{alpha_bar} in \cref{alg:KMM} we first compute trial step-lengths $\alpha^k_p = 0.99995\alpha_p^{*,k}$ and $\alpha^k_d = 0.99995\alpha_d^{*,k}$ and check whether the new iterates $x^{k+1} = x^{k} + \alpha_p^k \Delta x^k$ and $(y^{k+1},z^{k+1}) = (y^{k},z^{k}) + \alpha_d^k (\Delta y^k,\Delta z^k)$ remain in the neighborhood $\mathcal{N}$ and satisfy the descent property
\begin{equation*}
\left(x^{k+1}\right)^T z^{k+1} \leq (1 - \alpha^{*,k}(1-\beta_3))\left(x^k\right)^Tz^k
\end{equation*}
since by \cref{thm:compute_bar} this implies that the conditions on line \ref{enough} hold. If not, we compute $\bar{\alpha}^k$ and choose $\alpha_p^k = \alpha_d^k = \bar{\alpha}^k$.

\begin{figure}
	\centering
	\begin{tabular}{cc}
	\includegraphics[width=0.5\linewidth]{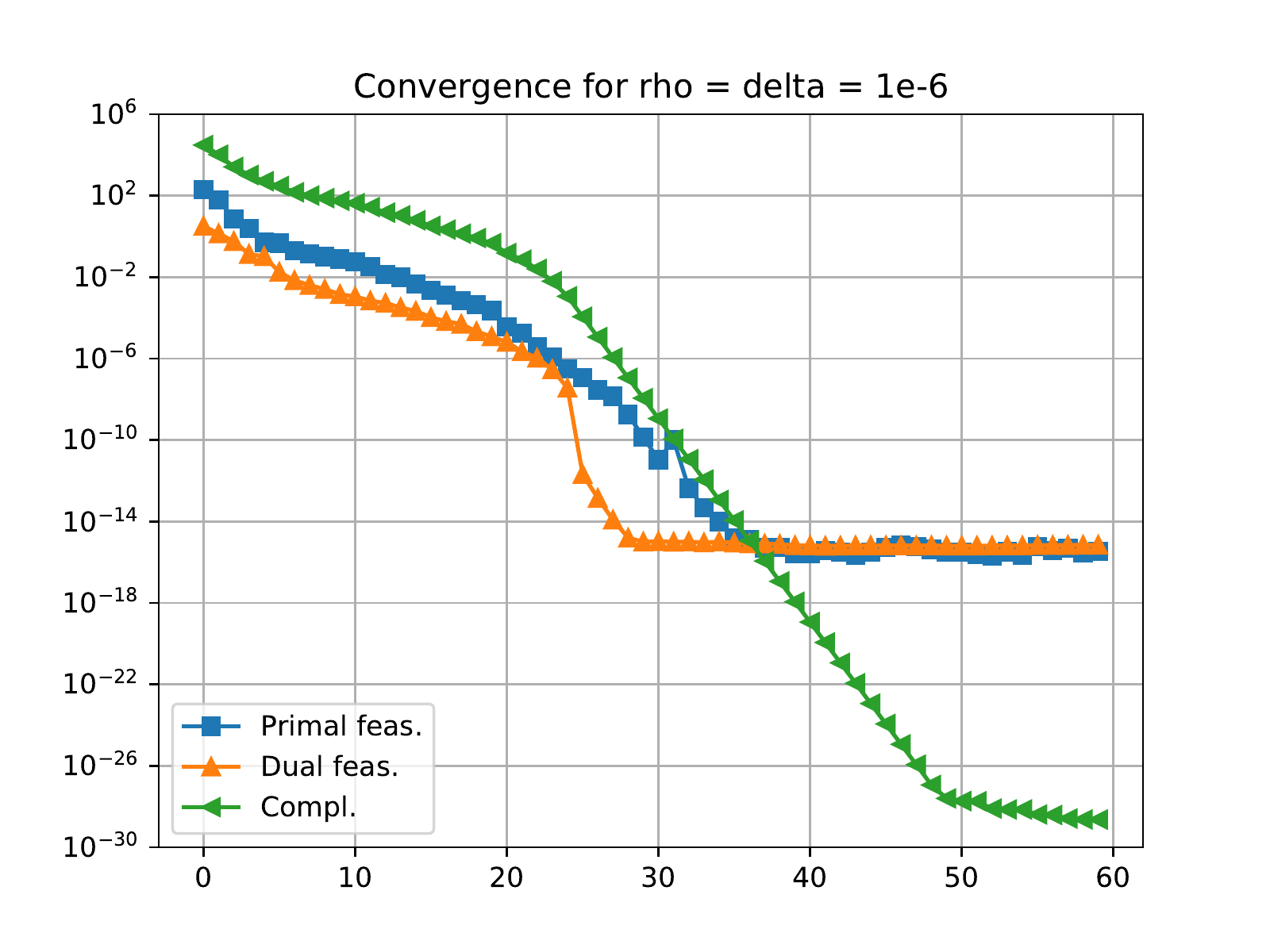} & \hspace{-1cm} \includegraphics[width=0.5\linewidth]{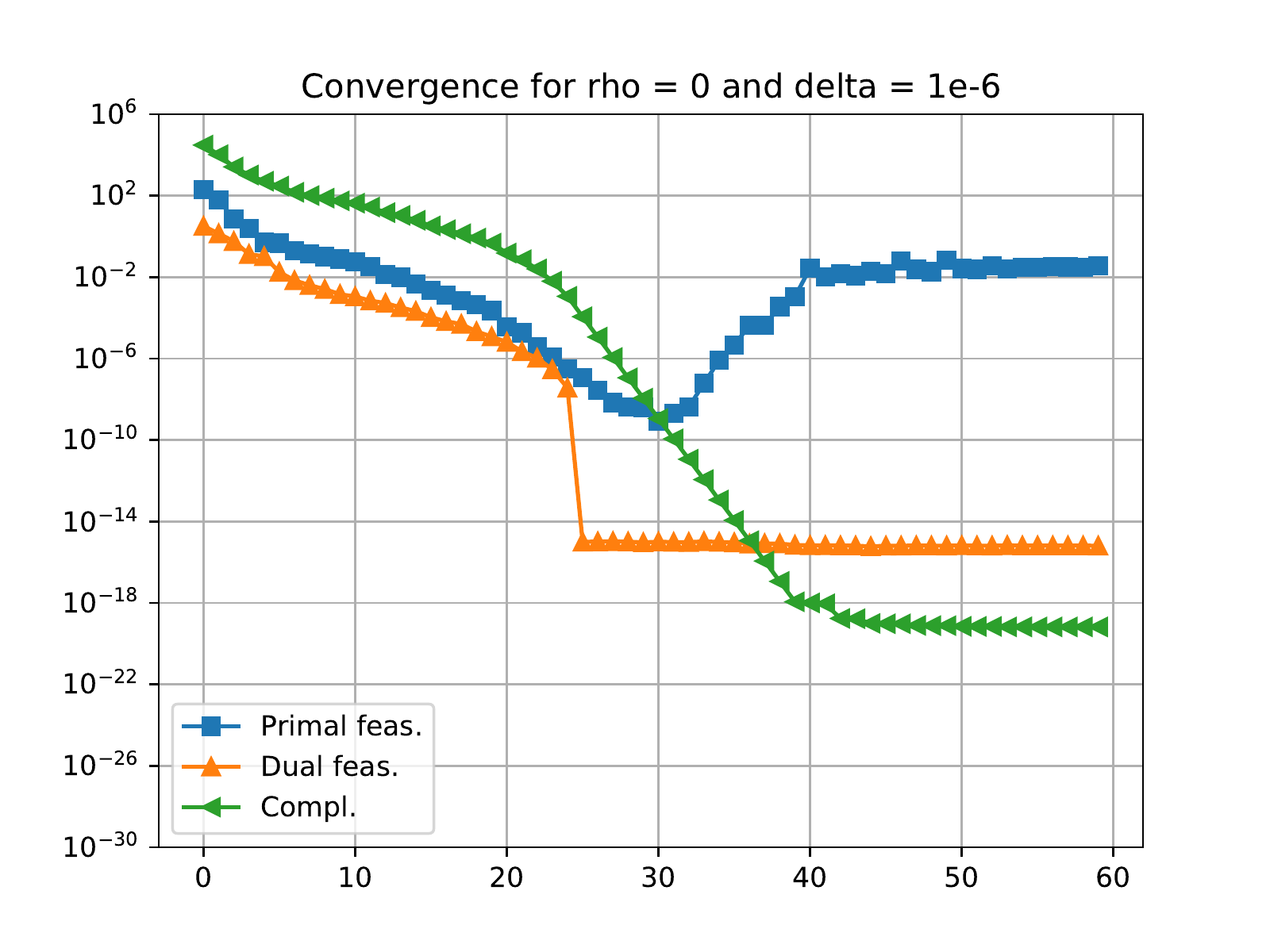}\\
	\end{tabular}
	\caption{Comparison of convergence of the interior-point method for two different choices of regularization parameters applied to NETLIB test-problem \texttt{ship12s}. Left: regularization parameters $\rho = \delta = 10^{-6}$, right: $\rho = 0$ and $\delta = 10^{-6}$.  \label{fig:compare_conv}}
\end{figure}

\begin{figure}
	\center
	\begin{tabular}{cc}
	\includegraphics[width=0.5\linewidth]{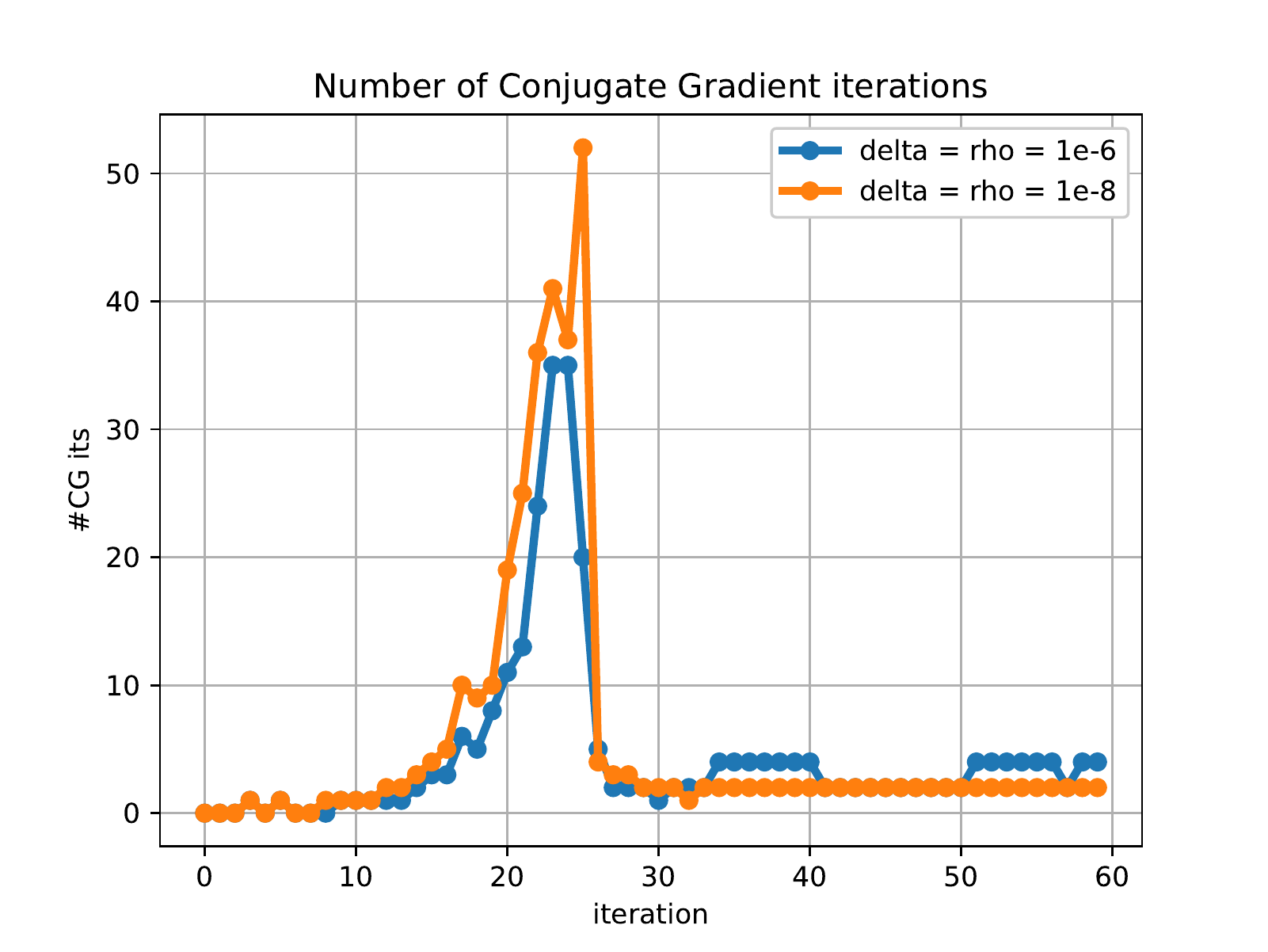} & \hspace{-1cm} \includegraphics[width=0.5\linewidth]{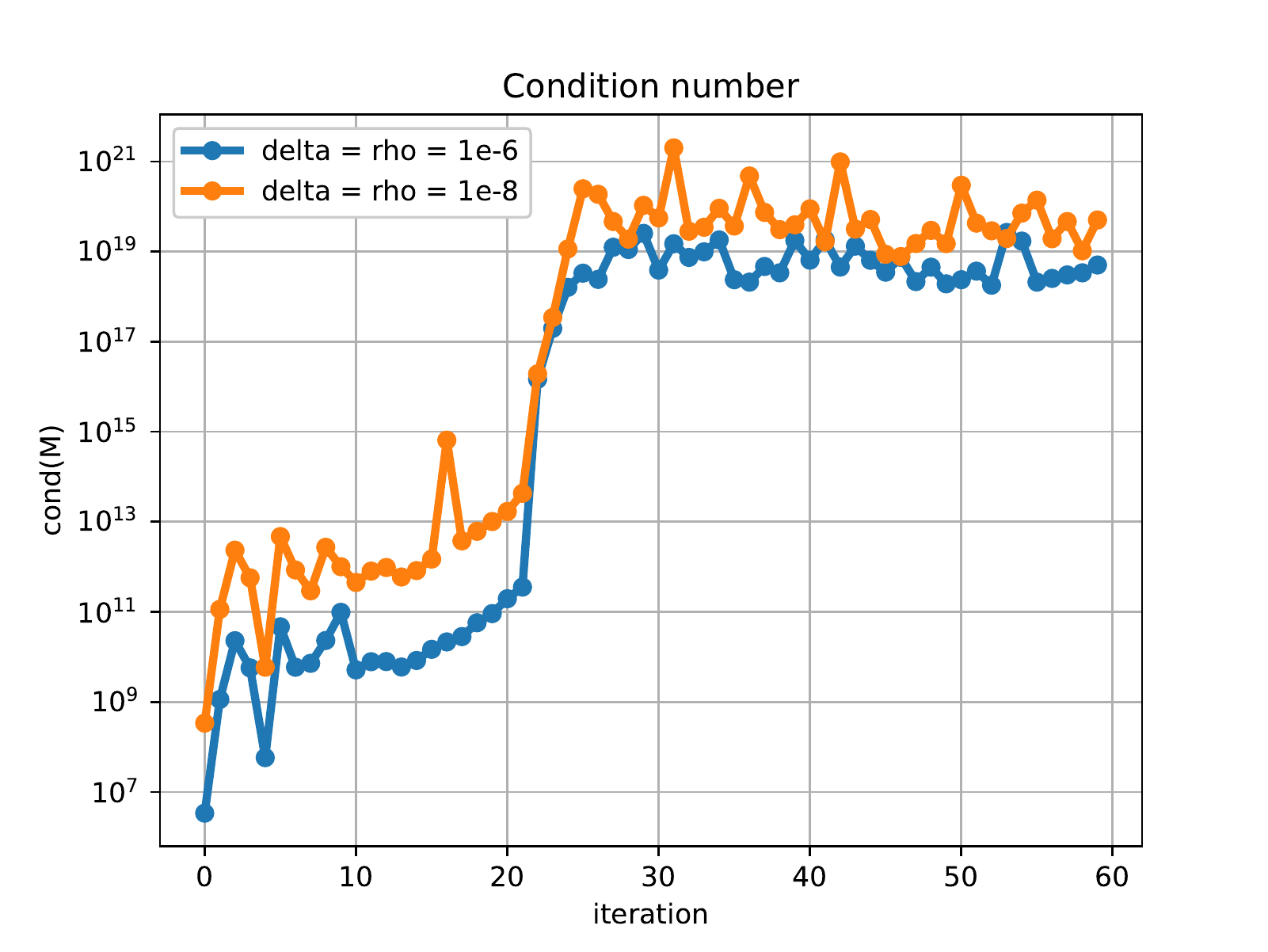}\\
	\end{tabular}
\caption{Left: number of Conjugate Gradients iterations needed in each interior-point iteration for two different choices of regularization parameters: $\rho = \delta = 10^{-6}$ and $\rho = \delta = 10^{-8}$. Right: condition number of the matrix $M = A(X^{-1}Z + \rho I)^{-1}A^T + \delta I$ for the different interior-point iterations. NETLIB test-problem \texttt{ship12s}. \label{fig:cg_its}}
\end{figure}

In our first experiment we illustrate the benefit of the regularization parameters $\rho$ and $\delta$. We apply \cref{alg:KMM} to NETLIB test-problem \texttt{ship12s}. The constraint matrix $A$ has dimensions $1151 \times 2869$ and contains $109$ rows with all zeroes. Hence, the matrix is significantly rank-deficient and we can not apply the algorithm with $\delta = 0$. Normally, we would simply remove these rows from the matrix, but for the purpose of this experiment we leave them in. We apply the algorithm once with $\rho =\delta = 10^{-6}$ and once with $\rho = 0$ and $\delta = 10^{-6}$ and let the algorithm run for a fixed number of iterations. The result is given by \cref{fig:compare_conv} where we show the convergence metrics $\|Ax^k - b\|/\|b\|$ for primal feasibility, $\|A^Ty + z - c\|/\|c\|$ for dual feasibility and $(x^k)^Tz^k /n$ for the complementarity condition. First of all we can clearly see that the algorithm converges nicely towards a primal-dual solution of the linear programming problems \cref{eq:primal,eq:dual} even though the matrix $A$ is rank deficient. Next, we can also clearly observe the benefit of the regularization term $\rho$. In the left figure, when using $\rho = 10^{-6}$ we have a very accurate solution in terms of all the convergence metrics shown, while the primal residual shows some very unstable behavior when $\rho = 0$. This is due to the fact that the matrix $X^{-1}Z$ has components converging either to $0$ or $\infty$ due to the complementarity condition, which causes numerical difficulties. 

In \cref{fig:cg_its} we report the condition number of the matrix and the number of Conjugate Gradient iterations needed to converge in each iteration of the interior-point method  when using $\rho = \delta = 10^{-6}$ and $\rho = \delta = 10^{-8}$. In the first number of iterations of the interior-point method we only need a very few CG iterations to satisfy the stopping criterion in \cref{alg:solver}. This is because of two reasons: first, the matrix $AD^{-1} A^T + \delta I$ is still relatively well-conditioned in the early iterations, and secondly, because the primal residual is still relatively large we have that the stopping criterion for the Conjugate Gradient method $\|(AD^{-1}A + \delta I)\Delta y - w \| \leq (1-\tau_1) \|Ax^k - b\|$ is quite loose. After a while, when the condition number starts to grow and the tolerance becomes more strict, we need to perform a lot more CG iterations. The effect is a bit more pronounced for the case $\rho=\delta=10^{-8}$ since the matrix is a little bit more ill-conditioned. It is also interesting to see that when the interior-point method has converged to double precision accuracy (around iteration 35 according to \cref{fig:compare_conv}) that we need to perform only a very few CG iterations, because in that case the solution $\Delta y$ is very close to zero, which is the initial iterate that we choose in CG. 

When the number of CG iterations becomes relatively large, the benefit of using a single precision factorization gets nullified by the additional computational cost of the CG iterations. In this case it would be more beneficial to switch from a single precision Cholesky factorization to simply using double precision for all computations. We can choose this `switch-point' based on a simple heuristic. In the ideal case we can expect that a Cholesky factorization and matrix-matrix multiplication as shown in \cref{alg:solver} is twice as fast in single precision than the same operations in double precision. Hence, we can expect a benefit of using the mixed-precision solver as long as the time it takes to perform the CG iterations does not exceed the time for constructing and factorizing the matrix in single precision. With this in mind modify our algorithm as follows. We monitor the CPU time of the construction of the matrix and the Cholesky factorization, as well as the run-time of the Conjugate Gradient method. As soon as the run-time of the Conjugate Gradient method exceeds $0.75\times$ the average time for the construction and factorization of the matrix, we switch to a full double-precision implementation. The factor $0.75$ is chosen here, because there is some additional overhead in practice when using the mixed-precision solver \cref{alg:solver}. We test this simple heuristic with the next experiment.

\begin{table}
\begin{center}
\begin{tabular}{l||rr|rr|rrrr}
						 &&& \multicolumn{2}{c|}{Double precision} & \multicolumn{4}{c}{Mixed-precision}  \\ 
						\textbf{Problem} & $m$ & $n$ & $\#$iter  & \text{time} & $\#$iter &  \text{time} & $\overline{\text{CG}}_\text{its}$ & $k_\text{switch}$ \\ \hline \hline 
						\texttt{bnl2}        & 2,324  & 4,486  & 58  &  22.93s & 58  & 17.28s & 3.18   & 48 \\
						\texttt{d2q06c}      & 2,171  & 5,831  & 52  &  19.87s & 52  & 16.34s & 3.69   & 31 \\
						\texttt{degen3}      & 1,503  & 2,604  & 30  &   3.79s & 30  & 3.15s  & 4.77   & 10 \\
			      \texttt{maros$\_$r7} & 3,136  & 9,408  & 25  &  29.21s & 25  & 19.08s & 2.12   & $\infty$ \\
						\texttt{qap12}       & 3,192  & 8,856  & 30  &  33.61s & 31  & 30.68s & 3.69   & 15 \\
						\texttt{sctap2}      & 1,090  & 2,500  & 23  &   1.47s & 23  & 0.97s  & 1.28   & 17 \\
						\texttt{sctap3}      & 1,480  & 3,340  & 23  &   2.91s & 23  & 2.03s  & 1.75   & 19 \\
						\texttt{ship12l}     & 1,151  & 5,533  & 30  &   3.43s & 30  & 2.77s  & 1.50   & 17 \\
						\texttt{ship12s}     & 1,151  & 2,869  & 30  &   2.54s & 30  & 2.06s  & 1.88   & 17 \\
						\texttt{stocfor2}    & 2,157  & 3,045  & 47  &  12.83s & 51  & 12.28s & 0.95   & 18 \\
						\texttt{truss}       & 1,000  & 8,806  & 31  &   4.47s & 31  & 3.61s  & 3.47   & 16 \\
						\texttt{woodw}       & 1,098  & 8,418  & 133 &  19.71s & 142 & 20.20s & 1.45   & 20 \\
						\texttt{cre$\_$a}    & 3,516  & 7,248  & 52  &  74.86s & 55  & 58.78s & 3.77   & 30 \\
						\texttt{cre$\_$c}    & 3,068  & 6,411  & 54 &   54.95s & 61  & 42.28s & 1.74   & 30 \\
\end{tabular}
\end{center}
\caption{Comparison of the mixed-precision interior-point method and a double precision implementation applied to a selection of NETLIB test-problems. We report the dimensions of the constraint matrix ($m,n$), the total number of interior-point iterations needed to converge ($\#$iter) and the total run-time (time) for both implementations. For the mixed-precision interior-point method we also report the average number of Conjugate Gradients iterations ($\overline{\text{CG}}_\text{its}$) and the switch-point ($k_\text{switch}$).} \label{table:table}
\end{table}

\begin{figure}
	\centering
	\begin{tabular}{cc}
	\includegraphics[width=0.49\linewidth]{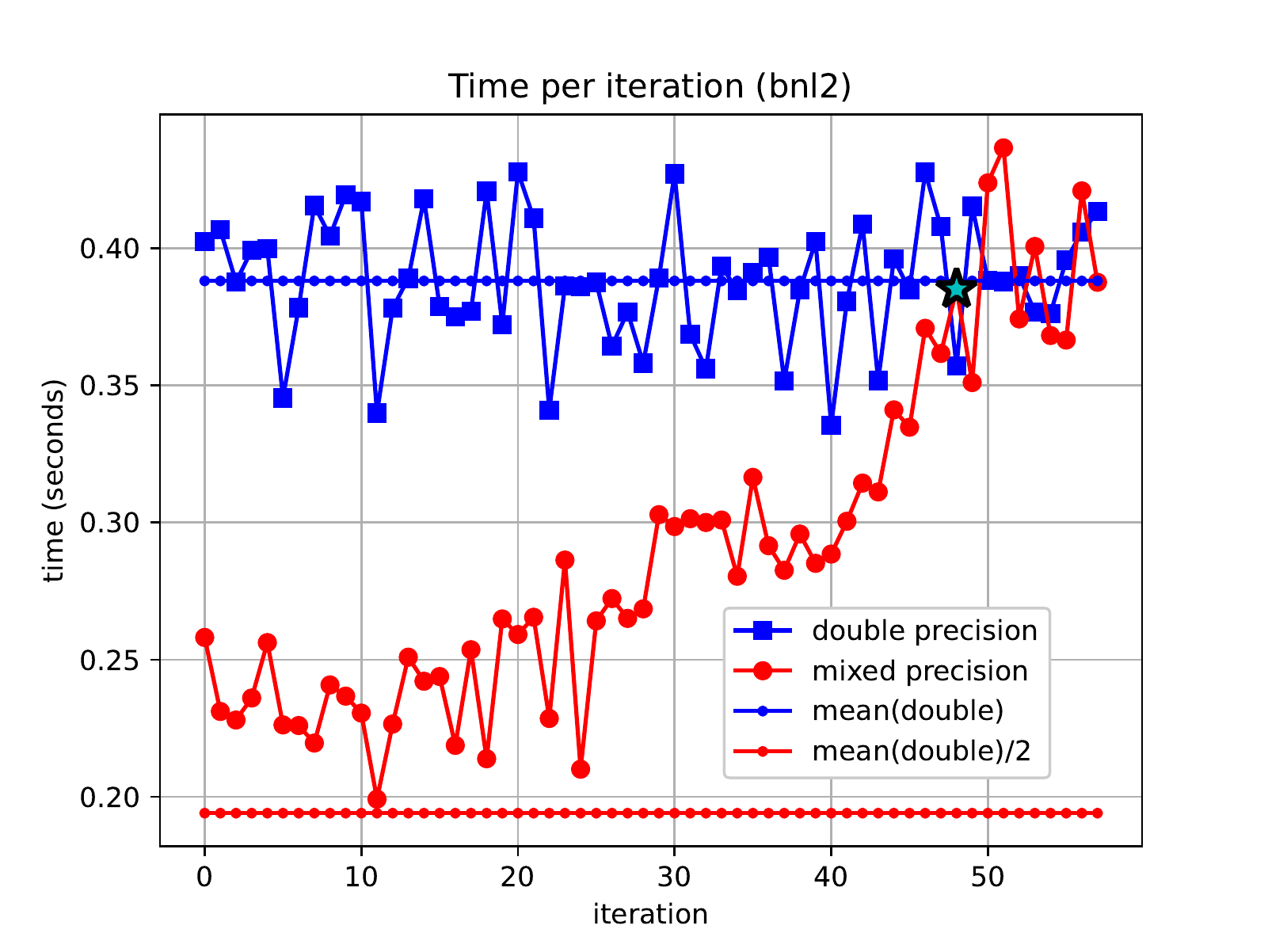} & \hspace{-1cm} \includegraphics[width=0.49\linewidth]{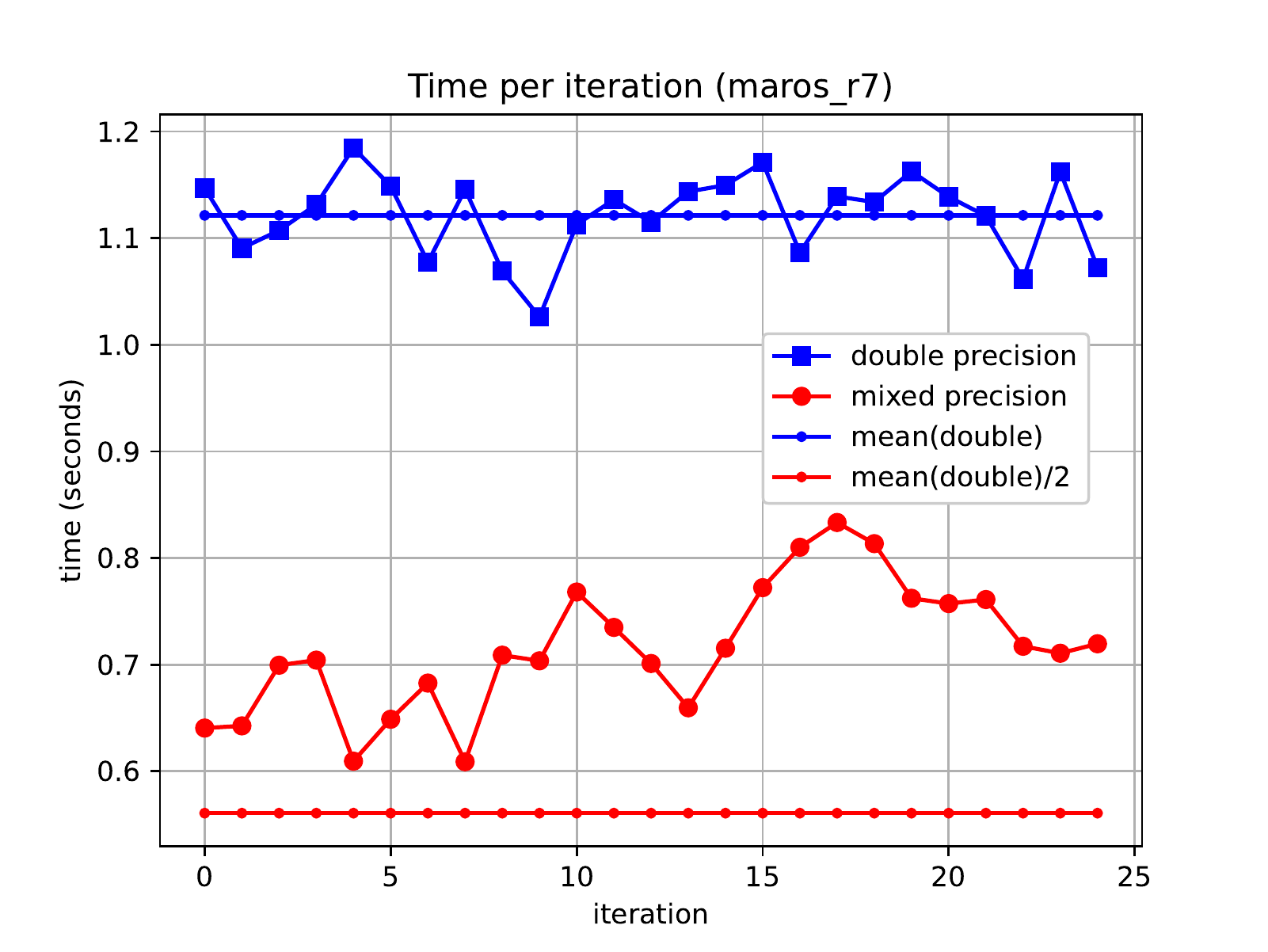} \vspace{-0.2cm}\\
	\includegraphics[width=0.49\linewidth]{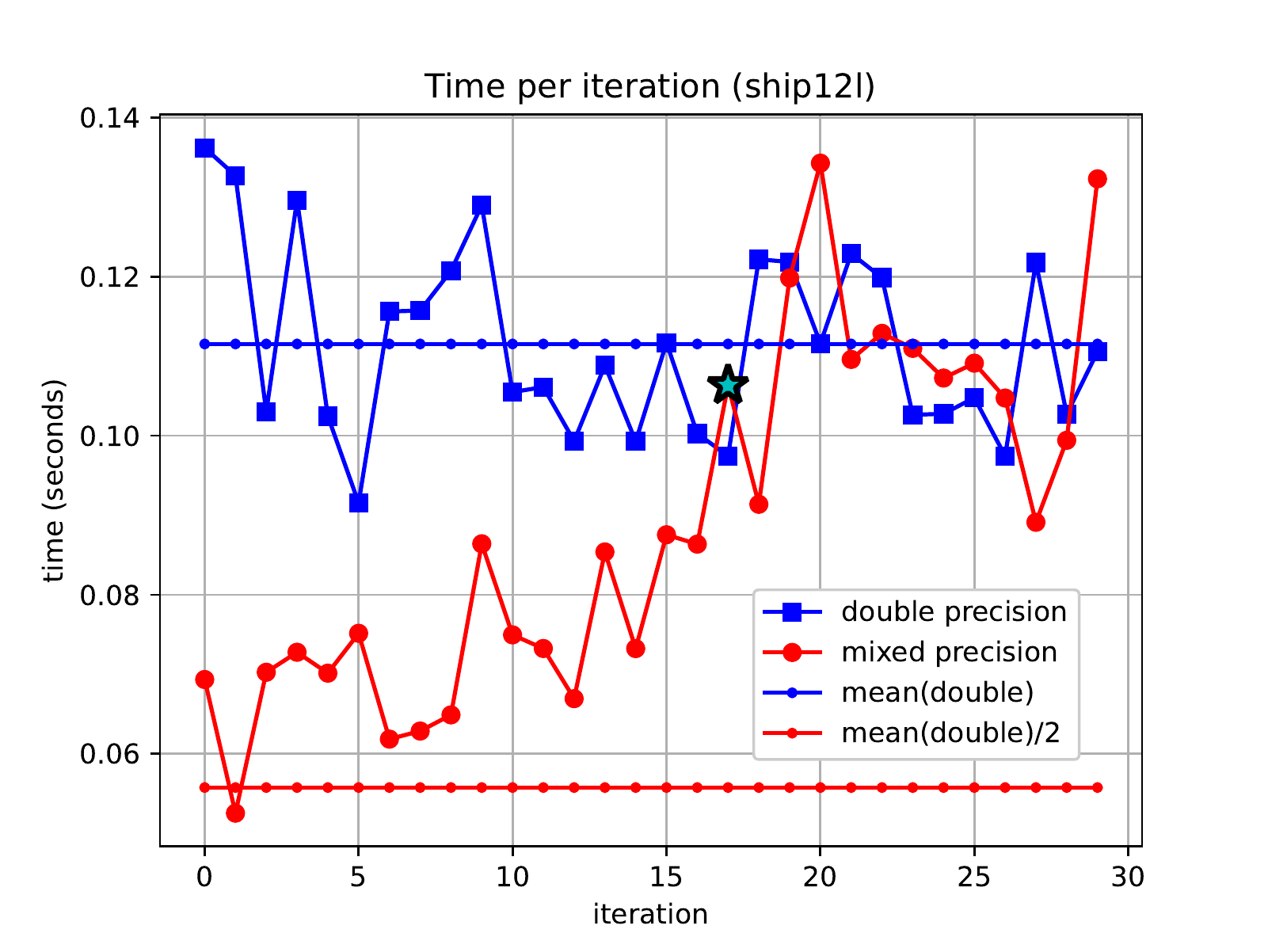} & \hspace{-1cm} \includegraphics[width=0.49\linewidth]{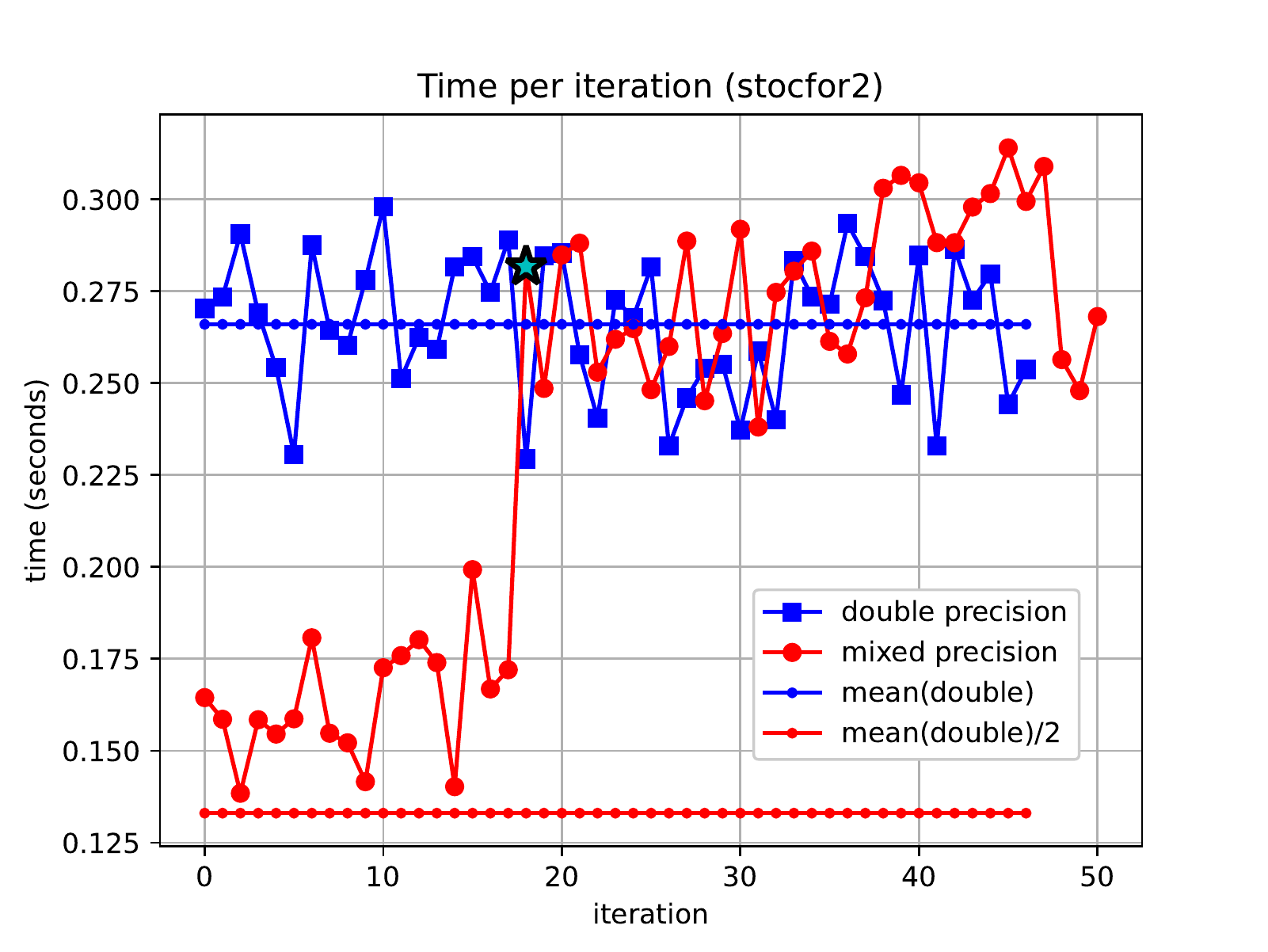} \vspace{-0.2cm}
	\end{tabular}
	\caption{The time per iteration for the mixed-precision interior-point method compared to a double precision implementation applied to a selection of NETLIB test-problems. The star indicates the switch-point $k_\text{switch}$ determined by the heuristic used in the mixed-precision implementation. \label{fig:timings}}
\end{figure}

For our final experiment we apply the regularized inexact interior-point method to the $14$ selected NETLIB test-problems, with the parameters as specified in the beginning of the section. We compare our implementation based on the mixed-precision solver \cref{alg:solver}, together with the switch-point explained in the previous paragraph and compare it with the `exact' counterpart where we use a double precision Choleksky factorization to solve the linear system of equations. We compare the number of (interior-point) iterations needed to converge, as well as the total run-time. For the mixed-precision solver we also report the average number of Conjugate Gradient iterations $\overline{\text{CG}}_\text{its}$ needed in each interior-point iteration and the switch-point $k_\text{switch}$ as determined by the heuristic. Note that $k_\text{switch}$ can vary slightly for different runs, since it is determined by actual timings during the execution. However, the overall performance and behavior does not vary much. We report the results in \cref{table:table}. 

The mixed-precision algorithm outperforms the double precision implementation in all test-problems except \texttt{woodw}. This is due to the fact that the switch-point happens very early in this example and a lot of interior-point iterations are needed for convergence. Interestingly, this is also the only test-problem where we need to compute $\bar{\alpha}^k$ in some of the iterations, see \cref{thm:compute_bar}. The benefit of the mixed-precision solver is rather limited for the test-problems where the switch-point occurs relatively early , see for example \texttt{stocfor2}. For test-problems where we can perform a lot of the Cholesky factorizations in single precision, it is clear that there is a significant reduction in the computational time. Note that for most of these problems, there is no difference in the total number of interior-point iterations needed to converge. 

In \cref{fig:timings} we show some more detailed timing results. In early iterations we can see that the mixed-precision solver requires slightly more than half the average time per iteration of the double precision implementation. Then, when the number of Conjugate Gradients iterations starts to increase, we can observe that the run-time of the interior-point method also increases. For most of these problems this increase is relatively gradual. At a certain point, the overhead of the Conjugate Gradients iterations is so large that the algorithm determines that we should switch to double precision. Note that this switch-point is chosen quite close to the average time per iteration of the double precision implementation, which is of course the ideal scenario. Hence, the heuristic seems to work quite well. 

\section{Conclusions and outlook}\label{sec:conclusions}

This work presents a convergence analysis of a regularized inexact interior-point method for solving linear programming problems. We illustrate the main benefits of regularization, namely that the fact that it alleviates the numerical difficulties arising from ill-conditioning of the matrices induced by the complementarity condition and that it allows us to easily solve problems with a rank deficient matrix. Another benefit of regularization that is not explored in detail in this work is the possibility of using a Cholesky-type factorization of the symmetric quasidefinite matrix arising from the augmented system approach. This is definitely worth exploring in future research since the augmented form is more suitable for sparse problems, especially when the matrix $A$ has one or more relatively dense column, since this results in a much denser matrix using the normal equations approach. The sparse Cholesky-factorization is also in general much more efficient than indefinite factorizations, since there is no need for pivoting to ensure numerical stability.

Solving the linear system of equations \textit{inexactly} in each iteration of the interior-point method offers a great number of possibilities in trying to improve computational performance of the method. The most notable and well-studied approach is the use of Krylov subspace methods combined with suitable preconditioners. The interior-point method developed in the current paper offers the opportunity to use many of the existing specialized linear algebra routines developed in other work. The main focus of these techniques has mostly been on linear programming problems with a sparse constraint matrix. However, we focus our attention on dense matrices. We develop a mixed-precision solver for the normal equations based on the Conjugate Gradient method preconditioned with a Cholesky factorization computed in IEEE single precision. We show that this leads to a significant reduction of the computational time for the majority of test-problems considered in the numerical experiments. We also touch on the possibility of exploiting IEEE half precision, which is expected to provide an even more substantial improvement. This is left a possible future work. 

\bibliographystyle{siamplain}
\bibliography{references}
\end{document}